\documentclass[11pt]{article}

\usepackage[T1]{fontenc}
\usepackage[utf8]{inputenc}
\usepackage[a4paper,margin=2.5cm]{geometry}

\usepackage{amsmath, amsthm, amssymb, mathtools}
\usepackage{enumerate}
\usepackage{graphicx}
\usepackage{xcolor}
\usepackage{tikz}
\usetikzlibrary{calc,trees,positioning,arrows,chains,shapes.geometric,%
	decorations.pathreplacing,decorations.pathmorphing,shapes,%
	matrix,shapes.symbols,plotmarks,decorations.markings,shadows}

\definecolor{lila}{RGB}{150,51,255}
\definecolor{verdesi}{RGB}{0,153,0}
\definecolor{linkblue}{RGB}{0,20,128}
\definecolor{citegreen}{RGB}{46,126,42}
\definecolor{urlmagenta}{RGB}{138,0,135}

\usepackage{hyperref}
\hypersetup{
	colorlinks = true,
	linkcolor  = linkblue,
	citecolor  = citegreen,
	urlcolor   = urlmagenta,
}

\theoremstyle{plain}
\newtheorem{theorem}{Theorem}[section]
\newtheorem{corollary}[theorem]{Corollary}
\newtheorem{lemma}[theorem]{Lemma}
\newtheorem{proposition}[theorem]{Proposition}

\theoremstyle{definition}

\theoremstyle{remark}
\newtheorem{remark}[theorem]{Remark}

\newcommand{\N}{\mathbb{N}}

\newcommand{\E}{\mathbb{E}}
\renewcommand{\P}{\mathbb{P}}
\newcommand{\VEC}[1]{\mathbf{#1}}
\newcommand{\ind}{\mathbf{1}}

\makeatletter
\newcommand{\leqnomode}{\tagsleft@true\let\veqno\@@leqno}
\makeatother

\begin{document}

\title{The $Q$-process for the general mother-dependent \\ multitype branching process}
\author{
	Airam Blancas\thanks{Department of Statistics, ITAM, Mexico.
		Email: \texttt{airam.blancas@itam.mx}}
	\and
	Mar\'ia Clara Fittipaldi\thanks{Facultad de Ciencias, UNAM, Mexico.
		Email: \texttt{mcfittipaldi@ciencias.unam.mx}}
	\and
	Sara\'i Hern\'andez-Torres\thanks{Instituto de Matem\'aticas, UNAM, Mexico.
		Email: \texttt{saraiht@im.unam.mx}}
}
\date{}
\maketitle

\begin{abstract}
	The general mother-dependent neutral mutations model is a multitype branching
	process describing a population evolving in discrete time, in which mutations
	occur among a finite set of alleles: each mutant, independently of the others,
	acquires an allele different from its mother's, according to a prescribed
	distribution.
	We study the associated mutant and clone-mutant chains conditioned on extinction
	in the distant future, that is, their $Q$-processes, and we give an explicit
	description of their laws in the (sub)critical and in the supercritical regimes
	alike.
	Even under this conditioning, the mother-dependent neutral mutations model is
	obtained from the corresponding infinite-alleles BGW process (studied in
	Bertoin (2010) and Blancas and Rivero (2018)) by an independent multinomial
	colouring with specific parameters.
	We also obtain the spine decomposition of the multitype mutant chain, from which
	we derive a Kesten--Stigum theorem in the supercritical regime and a genealogical
	construction of the $Q$-process in the (sub)critical one.
\end{abstract}



\section{Introduction}

Random fluctuations may allow a population to persist for long periods, even when extinction is inevitable, on time scales exceeding those accessible to observation.
Such behaviour arises, for example, in epidemic models; see~\cite{PenissonJacob12}.
In this work, we study the asymptotic behaviour of a population with neutral mutations and finitely many alleles, where the mutant population is conditioned on
long survival prior to extinction.

A classical model for population size is the Bienaym\'e--Galton--Watson (BGW) process, in which individuals reproduce independently with a common offspring distribution. 
For the BGW process $Z \coloneqq (Z_k,\, k \in \N_0)$, the extinction time is $\tau_Z \coloneqq \inf \{ k \in \N : Z_k = 0 \}$.
In general, the associated \emph{$Q$-process} $(Z_k^{\uparrow},\, k \in \N_0)$ is obtained by conditioning $Z$ to survive far into the future but eventually become extinct. 
Formally, for an initial population of $x \in \N_0$ individuals, its law is given by
\begin{equation} \label{eq:Qprocess}
	\P_x \big( Z_k^{\uparrow} = y \big)
	\coloneqq
	\lim_{n \rightarrow \infty } \P_x \big( Z_k = y \, \vert \, n + k < \tau_Z < \infty \big), \qquad k,y \in \N_0, \quad y \in \N_0. 
\end{equation}
Throughout this paper, we adopt this construction of the $Q$-process following~\cite[Chapter I.D]{AthreyaNey}.

Introducing independent mutations along ancestral lines enriches the BGW model by allowing offspring to differ in genotype while preserving the underlying
population-size dynamics. 
Under Kimura's infinite-alleles assumption~\cite{Kimura1971}, each mutation produces
a novel allele, leading to the \emph{BGW process with neutral mutations}~\cite{Bertoin09,Bertoin10}. In order to describe the sizes of allelic subfamilies together with their mutant descendants, 
Bertoin~\cite{Bertoin10} defined the \emph{infinite-alleles clone-mutant chain} $((T_k,M_{k+1}),\, k\in\N_0)$, which records the total population size, including clones and mutants, $(T_k,\, k \in \N_0)$, together with the number of mutants $(M_{k+1},\, k \in \N_0)$ across allelic generations. 
An individual belongs to the $k$-th allelic generation if there are exactly $k$ mutations along its genealogy. 
We refer to the latter process as the \emph{infinite-alleles mutant chain}, which is a BGW process as well.

The $Q$-process associated with the infinite-alleles clone-mutant chain was analyzed  in~\cite{BlancasRivero}. 
It is obtained as the weak limit under conditioning $((T_k,M_{k+1}),\, k\in\N_0)$ on survival up to a large but finite time $\tau \coloneqq \inf\{ k\in \N : M_{k} = 0 \}$, in analogy with~\eqref{eq:Qprocess}.

In this work, we consider a finite-alleles analogue of the setting above: each neutral mutation independently selects an allele in $[d]\coloneqq\{1,\dots,d\}$
according to a prescribed distribution, subject to the restriction that it differs from the type of its mother.
The resulting population is modelled by the \emph{general mother-dependent neutral
mutations model}, introduced in Section~\ref{sec:modelDef} as a $d$-type BGW process.
From now on, we will simply refer to it as the \emph{general mother-dependent model},
which is an extension of the model introduced in~\cite{BFH} and further studied in~\cite{BFH1}.

Extending the framework developed for Bertoin's BGW processes with neutral mutations, we study the associated $d$-type clone-mutant chain defined by a general mother-dependent model
\begin{equation} \label{eq:cloneMutantChainVector}
	\big( (\VEC{T}_k,\VEC{M}_{k+1}),\, k \in \N_0 \big),
\end{equation}
which takes values in $\N_0^d \times \N_0^d$ and records the type composition of both clones and mutants.

Our main result is Theorem~\ref{thm:condnonext_clonemut}, which establishes the existence of the associated $Q$-process
\[
	\big( (\VEC{T}^{\uparrow}_k,\VEC{M}^{\uparrow}_{k+1}),\, k \in \N_0 \big)
\]
for the $d$-type clone-mutant Markov chain and gives an explicit description of its law in terms of the original clone-mutant chain. We obtain a single formula that covers the (sub)critical and the supercritical regimes simultaneously.
The result holds under certain assumptions on the regime of the embedded BGW processes governing clones and mutants, which we specify below.

\subsection{Regimes and hypotheses} \label{sub:regimes_and_hypothesis}

In the general mother-dependent model, individuals in the population are either clones or mutants. Throughout this work, we impose assumption~\eqref{hyp:size}, under
which each clone lineage is subcritical. In particular, this excludes, almost surely, the possibility of an infinite lineage consisting entirely of clones.
The precise formulation of this assumption is given in
Subsection~\ref{sub:the_general_mother_dependent_model}.

The mutant population is encoded by the $d$-type mutant chain $\VEC{M} \coloneqq (\VEC{M}_k,\, k \in \N_0)$, which is a multitype BGW process by Proposition~\ref{prop:GTyMMarkov}.
We allow the mutant process to be subcritical, critical, or supercritical, provided
that extinction occurs with positive probability.
Also, the definition of the general mother-dependent model ensures that $\VEC{M}$ is positive regular, so its only absorbing state is $\VEC{0}$.

Let $\tau_\VEC{M}$ denote the extinction time of the mutant population, defined by
\begin{equation}\label{eq: extime}
	\tau_\VEC{M} \coloneqq \inf \{ k \in \N_0 : |\VEC{M}_k| = 0 \}.
\end{equation}
Throughout this work, we impose the following standing hypothesis:
\begin{equation} \label{hyp:timext}
	\P_{i}(\tau_\VEC{M}<\infty)>0, \qquad \text{for every } i\in [d]. \tag{H}
\end{equation}
Here and below, the subscript in the measure indicates the initial population; in this case, $\VEC{M}_0 = \VEC{e}_i$, meaning that the process starts with one individual of type $i$. 
By Proposition~\ref{prop: q} and Remark~\ref{re:Hcharacterisation} below,
\eqref{hyp:timext} holds if and only if $\mu^+(0)>0$.
Under hypothesis~\eqref{hyp:timext}, we further distinguish between different regimes.

In the subcritical and critical regimes, extinction occurs almost surely, that is, $\tau_\VEC{M} < \infty$ almost surely.
We treat these regimes jointly and refer to them as the \emph{(sub)critical case}, corresponding to the hypothesis
\begin{equation} \label{hyp:timextsubc} \tag{H1}
	\P_{i}(\tau_\VEC{M}<\infty)=1, \qquad \text{for every } i\in [d].
\end{equation}
By contrast, the supercritical regime for $\VEC{M}$ corresponds to the hypothesis
\begin{equation} \label{hyp:timextsupc} \tag{H2}
	\P_{i}(\tau_\VEC{M}<\infty) \in (0,1), \qquad \text{for every } i\in [d].
\end{equation}

While assumption~\eqref{hyp:size} is imposed throughout this work, our results are stated under hypothesis~\eqref{hyp:timext}; the distinction
between~\eqref{hyp:timextsubc} and~\eqref{hyp:timextsupc} enters only through the
value of the extinction probability $q$ introduced in Proposition~\ref{prop: q}, and we point out explicitly the (few) places where the two regimes must be treated
separately.

\subsection{Main results}

A key observation, leading to our main result Theorem~\ref{thm:condnonext_clonemut}, is that the transition probabilities of the clone-mutant chain depend only on those of the mutant chain. 
This property was established in~\cite[Lemma~1]{Bertoin10} for the infinite-alleles case and in~\cite[Lemma~2.2]{BFH1} for the (uniform) finite-alleles case.

Consequently, our analysis begins with the study of the $d$-type mutant chain $(\VEC{M}_k,\, k \in \N_0)$ conditioned on non-extinction, following an analogous
construction to~\eqref{eq:Qprocess}.
Since this chain is itself a $d$-type BGW process, Theorem~\ref{thm: Qprocess} defines its associated $Q$-process $(\VEC{M}^\uparrow_k,\, k \in \N_0)$ under the sole hypothesis~\eqref{hyp:timext}. Furthermore, Theorem~\ref{thm: Qprocess} expresses its transition probabilities explicitly in terms of those of the original mutant chain, of the extinction probability $q$, and of the derivative at $q$ of the offspring generating function of the \emph{mutant-count
process} $(|\VEC{M}_k|,\, k \in \N_0)$, defined as the total number of mutants at each allelic generation.
Lemma~\ref{lem: eqMand|VECM|} links the finite-alleles setting with the infinite-alleles model by showing that, for the general mother-dependent model, the
process $(|\VEC{M}_k|,\, k \in \N_0)$ is equal in law to the infinite-alleles mutant chain $(M_k,\, k \in \N_0)$.
Corollary~\ref{cor: Conmuta} extends this correspondence to the associated $Q$-processes $(|\VEC{M}^\uparrow_k|,\, k \in \N_0)$ and $(M^\uparrow_k,\, k \in \N_0)$, again showing equality in law, in both regimes.

Moreover, Proposition~\ref{prop:colorChain} identifies the transitions of the $Q$-process $(\VEC{M}^\uparrow_k,\, k \in \N_0)$.
Specifically, the distribution of $\VEC{M}^\uparrow_1$ is obtained by first generating the mutant count $\vert \VEC{M}^\uparrow_{1} \vert \overset{\mathcal{L}}{=} M_1^\uparrow$ related with $(M_k,\, k \in \N_0)$, and then independently assigning alleles according to the prescribed distribution on the set of
alleles $[d]$.
We refer to this assignment procedure as a \emph{colouring}, which follows a multinomial distribution.
These relationships are summarised in Figure~\ref{fig:conmDiagram}.

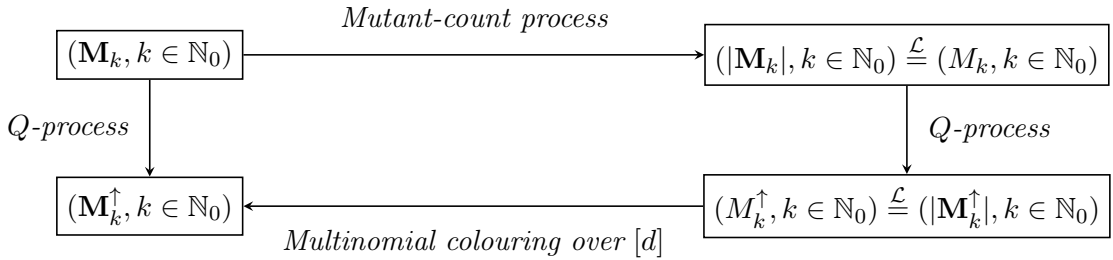
\begin{figure}[ht!]
	\centering
	\begin{tikzpicture}[-latex, auto, node distance=5.5cm, on grid, semithick]
		\node[draw] (A){$ (\VEC{M}_k,k\in \N_0)$};
		\node[draw] (B) [right=10 cm of A] {$ (|\VEC{M}_k|,k\in \N_0)\overset{\mathcal{L}}{=} (M_k,k\in \N_0)$};
		\node[draw] (C) [below=2 cm of A] {$(\VEC{M}^{\uparrow}_k,k\in \N_0)$};
		\node[draw] (D) [below=2 cm of B] {$ (M^\uparrow_k,k\in \N_0)\overset{\mathcal{L}}{=}(|\VEC{M}^\uparrow_k|,k\in \N_0)$};
		\path (A) edge [-stealth] node[above=0.15 cm] {{\it Mutant-count process}} (B);
		\path (B) edge [-stealth] node[right =0.15 cm] {{\it $Q$-process}} (D);
		\path (D) edge [-stealth] node[below =0.15 cm] {{\it Multinomial colouring over} $[d]$} (C);
		\path (A) edge [-stealth] node[left =0.15 cm] {{\it $Q$-process}} (C);
	\end{tikzpicture}
	\caption{Lemma~\ref{lem: eqMand|VECM|} shows that the general mother-dependent neutral mutations model is obtained by independently colouring the
	infinite-alleles mutant chain. 
    Corollary~\ref{cor: Conmuta} and Proposition~\ref{prop:colorChain} extend this correspondence to the associated $Q$-processes, in both the (sub)critical and the supercritical regimes, yielding a commutative diagram.}
	\label{fig:conmDiagram}
\end{figure}

Taken together, these results for the $d$-type mutant chain provide the key identities needed to prove our main result, Theorem~\ref{thm:condnonext_clonemut}, namely the existence of the associated $Q$-process
$( (\VEC{T}^{\uparrow}_k, \VEC{M}^{\uparrow}_{k+1}),\, k \in \N_0 )$ for the clone-mutant Markov chain.

Additionally, we study the spine decomposition of the $d$-type mutant chain, which sheds light on the chain in both regimes. 
In this construction, the spine corresponds to a distinguished lineage in the genealogical tree that survives throughout all generations.

In the supercritical case, the spine decomposition gives access to the asymptotic behaviour of the chain: Proposition~\ref{thm: KS-multitype} shows that, after normalization, the $d$-type mutant chain converges almost surely to the product of a non-negative random
variable $W$ and the deterministic vector $\VEC{u}$, where $W$ arises from the mutant-count process and $\VEC{u}$ is a deterministic probability vector giving the asymptotic proportions of types in the population. This provides another setting in which the diagram of Figure~\ref{fig:conmDiagram} applies, with the spine of the $d$-type mutant chain seen as an independent colouring of the spine of the infinite-alleles mutant chain.

On the other hand, in the (sub)critical case, the spine decomposition yields a construction of the $Q$-process, with Proposition~\ref{thm: spine subcritical}
describing the associated $Q$-process in terms of the size-biased distribution. 
We therefore obtain a construction of the genealogical tree associated with the $Q$-process of a (sub)critical $d$-type mutant chain.

\subsection{Notation} \label{sub:notation}

We write $\N=\{1,2,\dots\}$ and $\N_0=\N\cup\{0\}$, and $[d]$ for the set of alleles. Vectors in $S~\coloneqq\N_0^d$ are written in boldface; $\VEC{e}_i$ denotes the $i$-th canonical vector, $\VEC{0}$ and $\VEC{1}$ the vectors with all entries equal to $0$ and $1$, and $\VEC{v}(j)$ the $j$-th coordinate of $\VEC{v}$. 
We set
\[
	|\VEC{v}| \coloneqq \sum_{j=1}^d \VEC{v}(j),
	\qquad
	\overline{\VEC{v}}^{\,i} \coloneqq \VEC{v}-\VEC{v}(i)\VEC{e}_i,
	\qquad
	\VEC{s}^{\VEC{v}} \coloneqq \prod_{j=1}^{d} \VEC{s}(j)^{\VEC{v}(j)}
	\quad (\VEC{s}\in[0,1]^d),
\]
with the convention $0^0=1$, and we write $\VEC{u}\boldsymbol{\le}\VEC{v}$ for the
product (coordinatewise) order. 
Two subsets of $S$ occur repeatedly:
\[
	S^{\ast} \coloneqq \N_0^d\setminus\{\VEC{0}\},
	\qquad
	\overline{S}^{\,i} \coloneqq \{ \VEC{v}\in\N_0^d : \VEC{v}(i)=0 \},
	\qquad
	\overline{S}^{\,i,\ast} \coloneqq \overline{S}^{\,i} \cap S^{\ast}.
\]
$S^{\ast}$ will be the relevant state space for the mutant chain in the event that extinction does not occur.

For a multivariate probability generating function $\VEC{f}=(f^{(1)},\dots,f^{(d)})$ we denote by $\VEC{f}_n$ its $n$-fold iterate, $\VEC{f}_0(\VEC{s})=\VEC{s}$ and $\VEC{f}_{n+1}=\VEC{f}\circ \VEC{f}_n$, and we abbreviate $\VEC{f}_n^{\VEC{v}}(\VEC{s}) \coloneqq \big( \VEC{f}_n(\VEC{s}) \big)^{\VEC{v}}$.
Finally, $\P_{\VEC{a}}$ denotes the law of the model started from $\VEC{a}\in S$, and we abbreviate $\P_{i}\coloneqq \P_{\VEC{e}_i}$.

\subsection{Organization}

Section~\ref{sec:modelDef} introduces the BGW process with neutral mutations and the general mother-dependent model.
Section~\ref{sec:markov_chains_of_clones_and_mutants} analyzes the $d$-type clone-mutant, the $d$-type mutant, and the mutant-count Markov chains, primarily
through their probability generating functions.
After these preliminaries, Section~\ref{sec: backmulticond} presents the classical results on conditioned $d$-type BGW processes and derives the law of the $Q$-process of the mutant chain.
Section~\ref{sec:clone_mutant_nonExtinction} establishes the existence of the $Q$-process of the clone-mutant chain.
Section~\ref{sec: spine} develops the spine decomposition of the $d$-type mutant chain.
We use it to derive a multitype Kesten--Stigum theorem (see~\cite{KLPP}) for our model in the supercritical regime and to provide a probabilistic interpretation,
complementary to that of Section~\ref{sec: backmulticond}, of the $Q$-process
$(\VEC{M}_k^{\uparrow},\, k \in \N_0)$ in the (sub)critical regime. Finally, Appendix~\ref{app: proofpgf T_0Mvec} contains the proof of Proposition \ref{prop: genlawT_0M_1}, and  Appendix~\ref{app:proof4} contains the proof of Theorem~\ref{thm: Qprocess}.

\section{Models for neutral mutations along genealogies} \label{sec:modelDef}

We introduce two models: the BGW process with neutral mutations, defined in~\cite{Bertoin10}, and the general mother-dependent model.
Both share the same genealogical structure, namely a BGW process with superimposed mutations, and differ only in the mechanism assigning alleles to mutant individuals.

We begin by describing the evolution of the population size, which is common to both models.
Let $Z^{(+)} \coloneqq (Z^{(+)}_k,\, k \in \N_0)$ be a BGW process with offspring distribution $\mu^+\coloneqq(\mu^+(k),\, k \in \N_0)$.
This process describes the evolution of the total population size across generations.
Let $\zeta^{(+)}$ be a non-negative integer-valued random variable with law $\mu^+$, representing the total number of offspring of an individual.
We decompose $\zeta^{(+)}$ as \begin{equation}\label{eq: off+Ber}
	\zeta^{(+)} \coloneqq \zeta^{(c)} + \zeta^{(m)},
\end{equation}
where $\zeta^{(c)}$ and $\zeta^{(m)}$ denote, respectively, the numbers of clone and mutant offspring.

We exclude the degenerate cases $\zeta^{(c)} \equiv 0$ and $\zeta^{(m)} \equiv 0$, and assume $\mu^+(0) + \mu^+(1) < 1$ to rule out trivial situations.
Each child is a clone of its mother with probability $1 - r$ and a mutant with probability $r$, independently of all other individuals and of the process $Z^{(+)}$.

\subsection{The BGW process with neutral mutations}

We first consider the BGW process with neutral mutations, with mutation probability $r \in (0,1)$.
In this model, each mutation produces a novel allele, so it suffices to specify, for each individual, the numbers of clone and mutant offspring.
A key element of the model is therefore the pair $(\zeta^{(c)}, \zeta^{(m)})$, whose
joint distribution $\{ \pi_{k,l} : k, l \in \N_0 \}$ is given by
\begin{equation} \label{eq:lawreptypnot}
	\pi_{k,l} \coloneqq \P( \zeta^{(c)} = k,\, \zeta^{(m)} = l )= \mu^+(k+l) \binom{k+l}{k} (1 - r)^k r^l, \qquad k, l \in \N_0,
\end{equation}
where the last equality is a consequence of the independence of mutation events.

\subsection{The general mother-dependent model} \label{sub:the_general_mother_dependent_model}

We now define the general mother-dependent model.
Recall that $r \in (0,1)$ denotes the mutation probability and let $[d]$ be the set of possible alleles.
Mutations are assumed to change type: a mutant offspring always has a type different from that of its mother.

Once a mutation occurs, the type is determined by a \emph{mutation matrix} $\VEC{P}$, where $\VEC{P}_{ij}$ is the probability that a mutant child of a type $i$ individual has type $j$.
We assume that the mutation matrix $\VEC{P}$ satisfies $\VEC{P}_{ii} = 0$ for all $i \in [d]$. 
This condition enforces that mutations change type.
We assume that $\VEC{P}$ is aperiodic and irreducible, so that every type can eventually produce descendants of every other type. Under these assumptions, we refer
to $\VEC{P}$ as a \emph{mother-dependent mutation matrix}.
Since $\VEC{P}$ is an irreducible stochastic matrix, the Perron--Frobenius theorem gives that its maximal eigenvalue $\gamma$ equals $1$, with right eigenvector $\VEC{1}$ and left eigenvector $\VEC{u}$, the stationary distribution of $\VEC{P}$; we normalize
$\VEC{u}$ so that $\VEC{u}\cdot\VEC{1}=1$.

Mutations are neutral in the sense that all individuals, regardless of type, share the same total offspring distribution $\mu^+$.
Given a mutation probability $r \in (0,1)$ and a mother-dependent mutation matrix $\VEC{P}$, for each $i\in[d]$ let $\boldsymbol{\xi}^i(j)$ denote the number of offspring of type $j$ produced by an individual of type $i$, and let $\boldsymbol{\mu}_i$ denote its joint law.
For $\VEC{v} \in S$, this law satisfies
\begin{equation}\label{eq: mothdepproba}
	\begin{split}
		\boldsymbol{\mu}_i(\VEC{v})
		&= \P ( \boldsymbol{\xi}^{i}=\VEC{v})
		= \P \big( (\boldsymbol{\xi}^{i}(1),\dots,\boldsymbol{\xi}^{i}(d))=(\VEC{v}(1),\dots,\VEC{v}(d))\big)\\
		&= \mu^+(|\VEC{v}|) \binom{|\VEC{v}|}{\VEC{v}(1),\dots,\VEC{v}(d)} (1-r)^{\VEC{v}(i)} \prod_{j \neq i}\big(r \VEC{P}_{ij}\big)^{\VEC{v}(j)}.
	\end{split}
\end{equation}
Indeed, each child mutates independently with probability $r$, and its type is chosen
according to $\VEC{P}$.

We write $\boldsymbol{\mu}=(\boldsymbol{\mu}_i,\, i\in[d])$ for the family of offspring distributions and denote by $\P_{\VEC{a}}$ the law of the resulting $d$-type BGW process $\VEC{Y}$ started from an initial population $\VEC{a} \in S^{\ast}$.
We call this process the \emph{general mother-dependent (neutral mutations) model}, with total offspring distribution $\mu^+$, mutation rate $r$, and mother-dependent mutation matrix $\VEC{P}$. For simplicity, we omit the dependence on $\mu^+$, $r$ and
$\VEC{P}$ in this notation. 

Since mutant offspring necessarily have a type different from that of their mother, for each $i \in [d]$ the numbers of clone and mutant offspring of a type $i$ individual are
\begin{equation} \label{eq:xim}
	\xi_i^{(c)} = \boldsymbol{\xi}^{i}(i)
	\qquad \text{and} \qquad
	\xi_i^{(m)} = \sum_{ j \neq i } \boldsymbol{\xi}^{i}(j),
\end{equation}
respectively. 
By neutrality and independence of mutations, $\xi_i^{(m)}$ and $\xi_j^{(m)}$ are independent and identically distributed for all $i,j \in [d]$; let
$\mu^{(m)}$ denote their common distribution.

Let $\xi_i^{(+)}$ denote the total number of offspring of an individual of type $i \in [d]$:
\begin{equation} \label{eq:offspringdist1}
	\xi_i^{(+)}
	\coloneqq \sum_{j=1}^{d} \boldsymbol{\xi}^{i}(j)
	= \xi_i^{(c)} + \xi_i^{(m)}, \qquad i \in [d].
\end{equation}
Neutrality implies that $\xi_i^{(+)}$ has the same distribution for every type, and we write $\xi^{(+)} \overset{\mathcal{L}}{=} \xi_i^{(+)}$, with $\xi^{(+)} \sim \mu^+$.

From~\eqref{eq: mothdepproba}, it follows that for any $i \in [d]$, conditionally on $\{\xi_i^{(+)} = \ell\}$, the random variable $\boldsymbol{\xi}^{i}(i)$ is $\mathrm{Binomial}(\ell, 1 - r)$, while for each $j \neq i$, $\boldsymbol{\xi}^{i}(j)$ is $\mathrm{Binomial}(\ell, r \VEC{P}_{ij})$.
If $m \coloneqq \E (\xi_i^{(+)}) < \infty$, then taking expectations yields
\begin{equation} \label{eq: expectation}
	\E\big[ \boldsymbol{\xi}^{i}(i)\big] = (1 - r)m,
	\qquad
	\E\big[ \boldsymbol{\xi}^{i}(j)\big] = m r \VEC{P}_{ij}, \quad j \neq i.
\end{equation}

We can now state the first assumption indicated in
Subsection~\ref{sub:regimes_and_hypothesis}. Throughout this work, we assume
\begin{equation} 
	\E\!\big[\boldsymbol{\xi}^{i}(i)\big] < 1
	\quad \Longleftrightarrow \quad
	m < \frac{1}{1 - r},
	\tag{A} \label{hyp:size}
\end{equation}
where the equivalence follows from~\eqref{eq: expectation}.
In particular, assumption~\eqref{hyp:size} requires that $\mu^+$ has a finite first moment.

\begin{remark}\label{re: asspr}
	A classical assumption for multitype BGW processes is positive regularity. 
    As defined in~\cite[Section 2, Chapter V]{AthreyaNey}, given a multitype BGW process
	$\VEC{Y}$ with types $[d]$, we let $\VEC{m}^n = (m_{ij}^n)_{i,j \in [d]}$ be defined
	by $\E \big[ \VEC{Y}_n \vert \VEC{Y}_0 \big] = \VEC{Y}_0 \VEC{m}^n$.
	A process is \emph{positive regular} if there exists $n$ such that $m^n_{ij} > 0$ for all $i, j \in [d]$, in which case we say that the mean matrix
	$\VEC{m} = \VEC{m}^1$ is \emph{strictly positive}.
	The expectations in~\eqref{eq: expectation} show that the mean matrix is strictly positive if and only if $\VEC{P}$ is irreducible and aperiodic. By definition of a mother-dependent mutation matrix, this assumption always holds for the general
	mother-dependent model.

As the assumption that the mutation matrix $\VEC{P}$ is irreducible and aperiodic is what makes the classical multitype machinery available, dropping it leads to a different picture, in which the set of alleles decomposes into communicating classes and the surviving allelic composition is dictated by the
recurrent classes reachable from the initial type.
\end{remark}

\begin{remark}
	The definition of a mother-dependent mutation matrix allows for different probability distributions depending on the type.
	We may consider the case where the mutation type is drawn from a common probability distribution $p = (p_j)_{j \in [d]}$, and define $\VEC{P}$ by
	\[
		\VEC{P}_{ii} = 0, \qquad \VEC{P}_{ij} = \frac{p_j}{1 - p_i}, \qquad \text{for different } i, j \in [d].
	\]
	The mother-dependent neutral mutation model studied in~\cite{BFH, BFH1} corresponds
	to the case where $(p_j)_{j \in [d]}$ is the uniform distribution, and the model defined in~\eqref{eq: mothdepproba} generalises it.
	In the uniform case, the associated matrix satisfies $\VEC{P}_{ij} = 1/(d - 1)$ for
	all $i \neq j$, so it is doubly stochastic.
	In this setting, conditionally on $\{\xi^{(+)} = \ell\}$, $\boldsymbol{\xi}^{i}(i)$ is $\mathrm{Binomial}(\ell, 1 - r)$, and each $\boldsymbol{\xi}^{i}(j)$ for
	$j \neq i$ is $\mathrm{Binomial}(\ell, r/(d - 1))$.
\end{remark}

\section{Markov chains of clones and mutants} \label{sec:markov_chains_of_clones_and_mutants}

In this section, we study three processes derived from the general mother-dependent model: the clone-mutant, mutant, and mutant-count Markov chains.
The first tracks the sizes of both clone and mutant populations; it is based on the corresponding chain in the uniform mother-dependent neutral mutation model~\cite{BFH1}, which itself builds on~\cite{Bertoin10}.
The latter two describe the evolution of mutants.

Let $\VEC{Y} \coloneqq(\VEC{Y}_n,\, n \in \N_0)$ be a general mother-dependent model as defined in Subsection~\ref{sub:the_general_mother_dependent_model}. 
Individuals reproduce according to an offspring distribution $\mu^+$, mutate with probability
$r \in (0,1)$, and their mutation types are governed by a mother-dependent mutation matrix $\VEC{P}$.

Although $\VEC{Y}$ records only the number of individuals of each type in each generation $n$ of the original genealogy, its construction induces a random coloured genealogical tree.
We briefly recall the main features of this construction.
Each individual in $\VEC{Y}_0$ is the root of a genealogical tree.
For simplicity, assume that $\VEC{Y}_0 = \VEC{e}_{i}$, so that the population starts from a single individual of type $i$, and hence we have a single tree.

In the associated tree, vertices represent individuals and colours represent types.
The initial individual of type $i$ is the root.
Each vertex is connected by one edge to its mother (except for the root) and by further edges to its offspring (see~\cite[Section 3.4]{BFH} for a detailed construction). 
A colour change along an edge corresponds to a mutation event (see Figure~\ref{fig:mother-dependent-tree}).

An individual $u$ belongs to the \emph{$k$-th allelic generation} if its ancestral line from the root contains exactly $k$ mutations.
\emph{Mutants of the $k$-th allelic generation} are those individuals whose type differs from that of their mother; each such mutant initiates an \emph{allelic subfamily}, consisting of itself and all of its descendants of the same type. 
In the coloured genealogical tree, allelic subfamilies correspond to maximal connected monochromatic components.

\tikzset{
	fannode/.style={circle, draw=white, fill=white, line width=3pt, minimum size=1cm},
	rorangenode/.style={circle, draw=black, fill=orange, line width=2pt, minimum size=1cm},
	orangenode/.style={circle, draw=black, fill=orange!25, line width=2pt, minimum size=1cm},
	rosanode/.style={circle, draw=black, fill=magenta!40, line width=2pt, minimum size=1cm},
	rrosanode/.style={circle, draw=black, fill=magenta, line width=2pt, minimum size=1cm},
	verdenode/.style={circle, draw=black, fill=verdesi!20, line width=2pt, minimum size=1cm},
	rverdenode/.style={circle, draw=black, fill=verdesi, line width=2pt, minimum size=1cm},
	lilanode/.style={circle, draw=black, fill=lila!25, line width=2pt, minimum size=1cm},
	rlilanode/.style={circle, draw=black, fill=lila, line width=2pt, minimum size=1cm}
}
\begin{figure}[ht]
	\centering
	\resizebox{.55\textwidth}{!}{%
	\begin{tikzpicture}[
		>=latex, glow/.style={%
			preaction={draw,line cap=round,line join=round,
				opacity=0.5,line width=15pt,#1}},glow/.default=green,
		transparency group,
		level 1/.style={sibling distance=69mm, level distance=20mm},
		level 2/.style={sibling distance=20mm},
		level 3/.style={sibling distance=12mm},
		level 4/.style={sibling distance=11mm},
		line/.style={edge from parent/.style={solid, magenta, ultra thick, draw}},
		line2/.style={edge from parent/.style={solid, cyan, ultra thick, draw}}
		]
		\node[rverdenode] { }
		child{node[verdenode] {}
			child{node[verdenode] {}
				child{node[rlilanode, glow] {}
				edge from parent node {\textcolor{red}{\textbf{\large X}} }
				}
			}
			child{node[rorangenode, glow] {}
				child{node[orangenode] {} }
				child{node[orangenode] {} }
			edge from parent node {\textcolor{red}{\textbf{\large X}} }
			}
			child{node[rorangenode, glow] {}
				child{node[rverdenode, glow=pink] {}
				edge from parent node {\textcolor{red}{\textbf{\large X}} }
				}
			edge from parent node {\textcolor{red}{\textbf{\large X}} }
			}
		}
		child{node[rorangenode, glow] {}
			child{node[rlilanode,glow=pink] {}
				child {node[rorangenode] {}
				edge from parent node {\textcolor{red}{\textbf{\large X}} }
				}
				child {node[lilanode] {} }
			edge from parent node {\textcolor{red}{\textbf{\large X}} }
			}
			child{node[orangenode] {}}
			child{node[rverdenode, glow=pink] {}
				child {node[verdenode] {} }
				edge from parent node {\textcolor{red}{\textbf{\large X}} }
			}
			child{node[rlilanode,glow=pink] {}
				child{node[lilanode] {} }
				child{node[lilanode] {} }
			edge from parent node {\textcolor{red}{\textbf{\large X}} }
			}
		edge from parent node {\textcolor{red}{\textbf{\large X}} }
		};
	\end{tikzpicture}}
	\caption{A representation of the general mother-dependent neutral mutations model with $d = 3$ types (green, orange, purple). 
    Mutations are marked on the edges; mutant individuals are drawn in a bolder shade. 
    Green rings indicate first-generation mutants; pink rings indicate second-generation ones.}
	\label{fig:mother-dependent-tree}
\end{figure}
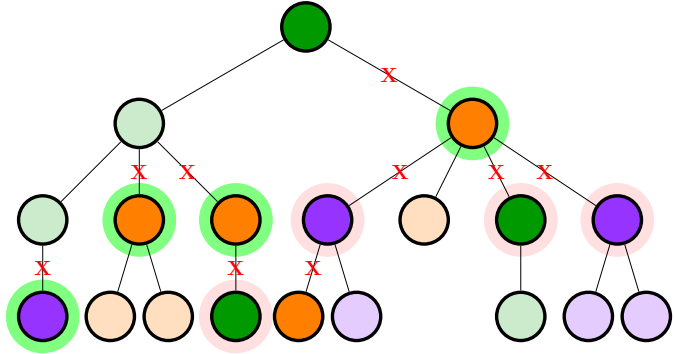

For $k \in \N_0$ and $i \in [d]$, let $\VEC{M}_k(i)$ be the number of mutants of type $i$ in the $k$-th allelic generation; each one of such individuals is the firstborn of a distinct allelic family. 
By convention, individuals at level $0$ are the mutants of the $0$-th allelic generation, so that $\VEC{M}_0 = \VEC{a}$ when the process starts
from $\VEC{a}$.

Let $\VEC{T}_k(i)$ be the total number of individuals of type $i$ in the $k$-th allelic generation, including both mutants and their clones, and set
\[
	\VEC{M}_k=( \VEC{M}_k(1), \dots , \VEC{M}_k(d)) \qquad
	\text{and} \qquad
	\VEC{T}_k=(\VEC{T}_k(1), \dots ,\VEC{T}_k(d)), \qquad k\in \N_0.
\]
Then $|\VEC{M}_k|$ is the total number of mutants in the $k$-th allelic generation, and $|\VEC{T}_k|$ is the total population size of that allelic generation.

The $d$-type \emph{clone-mutant chain} associated with the general mother-dependent model under $\P_{\VEC{a}}$ is the process $\big( (\VEC{T}_k,\VEC{M}_{k+1}),\, k \in \N_0 \big)$.
The \emph{$d$-type mutant chain} is $\big( \VEC{M}_k,\, k \in\N_0 \big)$, and the \emph{mutant-count chain} is $\big( \vert \VEC{M}_k \vert,\, k\in \N_0 \big)$.
Observe that, under $\P_{i}$, all mutants of the first allelic generation have a type different from $i$, so that $\VEC{M}_1 \in \overline{S}^{\,i}$ almost surely; likewise $\VEC{T}_0 \in \N\VEC{e}_i$.
We begin by establishing that these processes are Markov chains.

\begin{proposition}\label{prop:GTyMMarkov}
	Consider a general mother-dependent model under $\P_{\VEC{a}}$.
	\begin{enumerate}
		\item[(i)] The $d$-type mutant chain $( \VEC{M}_k,\, k \in\N_0 )$ is a $d$-type BGW process with initial state $\VEC{M}_0 = \VEC{a}$ and offspring distribution
		$\P_{i}( \VEC{M}_1\in \cdot)$, $i \in [d]$.
		\item[(ii)] The mutant-count chain $( \vert \VEC{M}_k \vert,\, k \in\N_0 )$ is a BGW process with reproduction law $\mu^{(m)}$, defined below equation~\eqref{eq:xim}.
		\item[(iii)] The process $\big( (\VEC{T}_k,\VEC{M}_{k+1}),\, k \in \N_0 \big)$ is a
		Markov chain with transition probabilities satisfying 
        \begin{equation}\label{probatranTM}
			\P_{\VEC{a}}(\VEC{T}_k =\VEC{x},\VEC{M}_{k+1}=\VEC{y} \mid \VEC{T}_{k-1}=\VEC{u},\VEC{M}_{k}=\VEC{v} )
			=
			\P_{\VEC{v}} (\VEC{T}_0=\VEC{x},\VEC{M}_1=\VEC{y}),
		\end{equation}
		for all $\VEC{u},\VEC{v}, \VEC{x}, \VEC{y}\in\N_0^d$ with
		$\VEC{v} \boldsymbol{\le} \VEC{x}$ in the product order.
	\end{enumerate}
\end{proposition}

Since the construction of the associated $d$-type BGW tree with offspring distribution~\eqref{eq: mothdepproba} is identical to that in the uniform case
of~\cite{BFH1}, the Markov property in Proposition~\ref{prop:GTyMMarkov} follows
directly from the general branching property of~\cite[Equation 1]{jagers2008}.
Proposition~\ref{prop:GTyMMarkov} is therefore a straightforward generalisation of~\cite[Lemma 3.2]{BFH1}, which in turn extends~\cite[Lemma 1]{Bertoin10} to the finite-type mutation framework. 
\begin{remark}\label{re: sameextime}
    Note that if $\VEC{M}_0=\VEC{0}$, then $\VEC{T}_0=\VEC{0}$. Hence, using the transition probabilities \eqref{probatranTM}, $\VEC{M}_k=\VEC{0}$ implies $\VEC{T}_k=\VEC{0}$. On the other hand, as by construction $|\VEC{M}_k| \leq |\VEC{T}_k|$, we have that $\VEC{T}_k=\VEC{0}$ implies $\VEC{M}_k=\VEC{0}$. Therefore 
    \begin{equation*}
        \tau_{(\VEC{T},\VEC{M})}\coloneqq\inf\{k \in \N_0 : (\VEC{T}_{k},\VEC{M}_{k+1})=\VEC{0} \}=\tau_{\VEC{M}},
    \end{equation*}
    where $\tau_{\VEC{M}}$ is defined in \eqref{eq: extime}.
\end{remark}

The main goal of this section is to describe the transitions of the clone-mutant chain and their dependency  only on the $d$-type mutant chain via their respective probability generating functions (p.g.f.s).
Most results in this section directly extend the analysis of the clone-mutant Markov
chain in~\cite{BFH1} to the non-uniform setting.

For a fixed type $i \in [d]$, the p.g.f. of $(\VEC{T}_0, \VEC{M}_{1})$ starting from
$\VEC{M}_0=\VEC{e}_i$ is
\begin{equation}\label{eq:MGFone}
\varphi^{(i)}( x, \VEC{y})
	\coloneqq
	\E_{i}\left[ x^{ \VEC{T}_{0}(i) }\prod_{j\neq i} \VEC{y}(j)^{ \VEC{M}_1(j) } \right],
	\qquad x \in [0,1],\ \VEC{y} \in [0,1]^d;   
\end{equation}

note that $\varphi^{(i)}(x,\VEC{y}) = \varphi^{(i)}(x,\overline{\VEC{y}}^{\,i})$, since the
product involves only the coordinates $j\neq i$.
We then define
\[
	\boldsymbol{\varphi}( \VEC{x}, \VEC{y})
	\coloneqq
	\big( \varphi^{(1)}(\VEC{x}(1),\overline{\VEC{y}}^{\,1}),\dots, \varphi^{(d)}(\VEC{x}(d),\overline{\VEC{y}}^{\,d}) \big),
	\qquad (\VEC{x},\VEC{y}) \in [0,1]^d \times [0,1]^d.
\]
Let $g^{(i)}$ be the p.g.f. of the offspring distribution $\boldsymbol{\mu}_i$ for
individuals of type $i$ given in~\eqref{eq: mothdepproba}, that is,
\begin{equation}\label{eq: g_i}
	g^{(i)}(\VEC{s})
	\coloneqq
	\E_{i}\left[ \VEC{s}(i)^{ \boldsymbol{\xi}^{i}(i) } \prod_{j\neq i} \VEC{s}(j)^{ \boldsymbol{\xi}^{i}(j) } \right]
	= \sum_{\VEC{v} \in \N_0^d} \VEC{s}^{\VEC{v}}\, \boldsymbol{\mu}_i (\VEC{v}),
	\qquad \VEC{s} \in [0,1]^d.    
\end{equation}
We shall write $g^{(i)}(\VEC{s})=g_{i}( \VEC{s}(i), \overline{\VEC{s}}^{\,i})$ to emphasise
the roles of clones and mutants.

The following proposition extends~\cite[Proposition 3.1]{BFH1} to the non-uniform setting. We reproduce the proof in Appendix~\ref{app: proofpgf T_0Mvec} for completeness.

\begin{proposition}\label{prop: genlawT_0M_1}
	Consider the general mother-dependent neutral mutations model under the law $\P_{i}$.
	Then the probability generating function of $(\VEC{T}_0(i), \VEC{M}_1)$ satisfies
	the equation
	\begin{equation}\label{eq: pgfTM_i}
		\varphi^{(i)}( x, \overline{\VEC{y}}^{\,i}) = x\, g^{(i)}\bigl(\varphi^{(i)}(x, \overline{\VEC{y}}^{\,i}), \overline{\VEC{y}}^{\,i}\bigr),
		\qquad x\in [0,1],\ \VEC{y}\in [0,1]^d;
	\end{equation}
	and the distribution of $(\VEC{T}_0(i), \VEC{M}_1)$ is given by
	\begin{equation}\label{eq: distT_0M_1}
		\P_{i}(\VEC{T}_0=k\VEC{e}_i,\, \VEC{M}_1=\boldsymbol{\ell})
		= \frac{1}{k}\,\boldsymbol{\mu}_i^{*k}(\boldsymbol{\ell} + (k-1)\VEC{e}_i),
		\qquad k \in \N,\ \boldsymbol{\ell} \in \overline{S}^{\,i},
	\end{equation}
	where $\boldsymbol{\mu}_i^{*k}$ denotes the $k$-th convolution power of
	$\boldsymbol\mu_i$, which corresponds to
	\begin{align}
		\boldsymbol{\mu}_i^{*k}(\boldsymbol{\ell} + (k-1)\VEC{e}_i)
		&= (\mu^{+})^{*k}(|\boldsymbol{\ell}| + k-1)\binom{|\boldsymbol{\ell}| + k-1}{k-1}(1-r)^{k-1} r^{\vert \boldsymbol{\ell} \vert}\binom{|\boldsymbol{\ell}|}{\boldsymbol{\ell}(1),\dots,\boldsymbol{\ell}(d)}\prod_{j\neq i}\big( \VEC{P}_{ij} \big)^{\boldsymbol{\ell}(j)}.
		\label{eq: mu_iconvo}
	\end{align}
\end{proposition}

The proof of Proposition~\ref{prop: genlawT_0M_1} follows the arguments
of~\cite[Appendix B]{BFH1}, which remain valid under hypothesis~\eqref{hyp:size}.
Furthermore, applying the arguments of~\cite[Appendix B.2]{BFH1} yields the analogue
of~\cite[Corollary 3.2]{BFH1} for the general model.
In particular, by combining this with~\eqref{eq: expectation}, we obtain the following
result.

\begin{corollary} \label{cor: cor 3.4}
	Assume~\eqref{hyp:size}. For each $i \in [d]$,
	\[
		\E_{i} \big[\VEC{T}_0(i)\big]
		= \frac{1}{1-m(1-r)}
		\qquad \text{and} \qquad
		\E_{i} \big[\VEC{M}_1(j) \big]
		= \frac{rm}{1-m(1-r)}\VEC{P}_{ij}
		\quad \text{for each } j \in [d] \setminus \{i\}.
	\]
\end{corollary}

Having established these preliminary results, we now turn to a simple (but useful)
relationship between $(\VEC{T}_0, \VEC{M}_1)$ and $\VEC{M}_1$ through their p.g.f.s:
\begin{equation}\label{eq: pgfM_i}
	f^{(i)} ( \VEC{y} ) \coloneqq
	\E_{i}\big[\VEC{y}^{\VEC{M}_1} \big] = \varphi^{(i)}( 1, \VEC{y}),
	\qquad \VEC{y}\in [0,1]^d.
\end{equation}
For simplicity of notation we will write $\VEC{f} \coloneqq (f^{(1)}, \dots, f^{(d)})$.

\subsection{Mutant-count Markov chain}

The goal of this subsection is to establish that the infinite-alleles mutant chain
$(M_k,\, k \in \N_0)$ of~\cite{Bertoin10} and the mutant-count process
$(|\VEC{M}_k|,\, k \in \N_0)$ defined above are equal in law.
As a first step, observe that the last two factors on the right-hand side
of~\eqref{eq: mu_iconvo} together yield the probability mass function of a multinomial
distribution. Hence, from~\eqref{eq: distT_0M_1}, the sum over all vectors
$\boldsymbol{\ell}$ with $\vert\boldsymbol{\ell}\vert=n$ simplifies to
\begin{equation*}
	\P_{i}(\VEC{T}_0=k\VEC{e}_i, |\VEC{M}_1|=n)
	=\frac{1}{k}(\mu^+)^{*k}(n + k-1)\binom{n + k-1}{k-1}(1-r)^{k-1}r^n .
\end{equation*}
By comparing the last equality with~\eqref{eq:lawreptypnot}, we deduce that
\[
	\P_{i}(\VEC{T}_0=k\VEC{e}_i, |\VEC{M}_1|=n)
	= \frac{1}{k}\pi^{*k}_{k-1,n},
\]
where $\pi^{*k}$ denotes the $k$-fold convolution of the measure $\pi$ given
in~\eqref{eq:lawreptypnot}.
Thus, regardless of the type of the initial individual, the right-hand side coincides
with the distribution of $(T_0, M_1)$ in~\cite[Proposition 1(ii)]{Bertoin10}.
Taking the marginal distribution then shows that the total number of mutants
$|\VEC{M}_1|$ in the finite-type mutation model has the same law as $M_1$ in the
infinite-type model started from a single individual (of any type).
By the branching property, we obtain the following result.

\begin{lemma}\label{lem: eqMand|VECM|}
	Starting from one individual, regardless of its type, we have that
	\[
		(|\VEC{M}_k|,\, k\in \N_0) \overset{\mathcal{L}}{=} (M_k,\, k\in \N_0).
	\]
	Thus $(|\VEC{M}_k|,\, k\in \N_0)$ is a single-type BGW process.
	Furthermore, for each $i \in [d]$ and $n\in\N_0$,
	\begin{equation}\label{eq:condtomcpmatix}
		\VEC{M}_1\big|_{\{\VEC{M}_0=\VEC{e}_i,\,|\VEC{M}_1|=n\}}\sim \mathcal{M}(n,\VEC{P}_{i\cdot}),
	\end{equation}
	where $\mathcal{M}$ denotes the multinomial distribution and $\VEC{P}_{i\cdot}$ the
	$i$-th row of the mutation matrix $\VEC{P}$.
\end{lemma}

Next, we aim to establish a relation between the mean of the mutant-count process
$(|\VEC{M}_k|,\, k\in \N_0)$ and the maximal eigenvalue of the mean matrix of
$(\VEC{M}_k,\, k\in \N_0)$, which characterise the behaviour of the process. To this
end, we introduce the p.g.f. of $\vert \VEC{M}_1 \vert$, given by
\begin{equation}\label{eq: pgf |m|}
	h(s) \coloneqq \E_{i}\left[s^{ |\VEC{M}_1| }\right], \qquad s\in [0,1],\ \forall\, i\in [d].
\end{equation}

\begin{lemma}\label{prop: mutbackfor}
	For each $i\in[d]$, the probability generating function $h$ can be expressed in terms
	of the generating function $f^{(i)}$ as
	\begin{equation*}
		h(s) = f^{(i)}(s\VEC{1}), \qquad s\in [0,1].
	\end{equation*}
	Conversely, given the probability generating function $h$ and the mother-dependent
	mutation matrix $\VEC{P}$, we recover $f^{(i)}$ by
	\begin{equation*}
		f^{(i)}(\VEC{y})= h\big( \big[\VEC{P}\VEC{y}\big]_i \big), \qquad \VEC{y} \in [0,1]^d,
	\end{equation*}
	where $\big[\VEC{P}\VEC{y}\big]_i$ denotes the $i$-th entry of the vector
	$\VEC{P}\VEC{y}$.
\end{lemma}

\begin{proof}
	Note that for any $i \in [d]$, the p.g.f. of $|\VEC{M}_1|$ is given by
	\begin{equation*}
		h(s)
		= \E_{i}\left[s^{ \sum_{j=1}^d \VEC{M}_1(j) } \right]
		=\E_{i}\left[\prod_{j=1}^d s^{ \VEC{M}_1(j) } \right]= f^{(i)}(s\VEC{1}).
	\end{equation*}
	On the other hand, we have that
	\begin{equation*}
		\begin{split}
			f^{(i)}(\VEC{y})
			&= \E_{i}\big[\VEC{y}^{\VEC{M}_1}\big] = \E_{i}\big[ \E_{i}\big( \VEC{y}^{\VEC{M}_1} \big| |\VEC{M}_1| \big) \big] \\
            &= \sum_{k=0}^\infty \P_{i}\big( |\VEC{M}_1|=k\big) \sum_{\VEC{z} \in \overline{S}^{\,i}} \P_{i}\big( \VEC{M}_{1}= \VEC{z} \big| |\VEC{M}_1|=k\big) \VEC{y}^{\VEC{z}}.
		\end{split}
	\end{equation*}
    
	Using Lemma~\ref{lem: eqMand|VECM|}, we have that
	\[
		\sum_{\VEC{z} \in \overline{S}^{\,i}}\P_{i}\big( \VEC{M}_{1} = \VEC{z} \big| |\VEC{M}_1|=k\big) \VEC{y}^{\VEC{z}}
		= \E_{i}\big[\VEC{y}^{\VEC{M}_1}\big| |\VEC{M}_1|=k \big] =\big[\sum_{j\neq i} \VEC{P}_{ij} \VEC{y}(j)\big]^{k},
	\]
	and from here
	\begin{equation*}
		f^{(i)}(\VEC{y})
		= \sum_{k=0}^\infty \P_{i}\big( |\VEC{M}_1|=k\big) \big[\sum_{j\neq i} \VEC{P}_{ij} \VEC{y}(j)\big]^{k}
		= h\big(\sum_{j\neq i} \VEC{P}_{ij} \VEC{y}(j)\big),
	\end{equation*}
	which is the announced identity, since $\VEC{P}_{ii}=0$.
\end{proof}
\begin{remark}\label{re : h pgf M}
	Lemmas~\ref{lem: eqMand|VECM|} and~\ref{prop: mutbackfor} together imply that $h$ is
	the p.g.f. of the offspring distribution of the BGW process $(M_k,\, k \in \N_0)$
	defined in~\cite{Bertoin10}.
	This function $h$ was further analyzed in~\cite[Section~2]{BlancasRivero}, where it
	played a key role in establishing the main results.
\end{remark}
The following proposition establishes the mean behaviour of the mutant chain.
\begin{proposition}\label{prop: meanmxrho}
	Assume~\eqref{hyp:size} and let $\VEC{A}$ be the mean matrix of the $d$-type
	BGW process $(\VEC{M}_k,\, k\in\N_0)$, that is,
	\[
		\VEC{A}_{ij}\coloneqq \E_{i}\big[\VEC{M}_1(j) \big],\qquad i,j \in [d].
	\]
	Then $\VEC{A}=\rho\,\VEC{P}$ and its maximal eigenvalue is
	\begin{equation}\label{eq: rho}
		\rho \coloneqq \frac{rm}{1-(1-r)m} = \E_{i}\big[|\VEC{M}_1|\big], \qquad i \in [d],
	\end{equation}
	with right eigenvector $\VEC{1}$ and left eigenvector $\VEC{u}$.
	Moreover, $( \VEC{M}_k,\, k \in\N_0 )$ is subcritical if and only if $m < 1$, and
	critical if and only if $m = 1$. Otherwise, the process is supercritical.
\end{proposition}

\begin{proof}
	By~\cite[Chapter V.2]{AthreyaNey} and Corollary~\ref{cor: cor 3.4}, the following
	equalities hold, respectively:
	\begin{equation}\label{eq: derivativef}
		\frac{\partial f^{(i)}(\VEC{y})}{\partial \VEC{y}(j)}\bigg|_{\VEC{y}=\VEC{1}}= \E_{i}\big[\VEC{M}_1(j) \big]= \frac{rm}{1-m(1-r)}\VEC{P}_{ij}.
	\end{equation}
	So $\VEC{A}_{ij}=\frac{rm}{1-m(1-r)} \VEC{P}_{ij}$ for all $i,j \in [d]$, and, since
	the maximal eigenvalue of $\VEC{P}$ is $\gamma=1$ with right eigenvector $\VEC{1}$
	and left eigenvector $\VEC{u}$, the maximal eigenvalue of $\VEC{A}$ is
	$\frac{rm}{1-m(1-r)}$, with the same eigenvectors.
	Moreover,
	\begin{equation*}
		h'(1)= \E_{i}\big[|\VEC{M}_1|\big]=\sum_{j=1}^d \frac{\partial f^{(i)}(\VEC{y})}{\partial \VEC{y}(j)}\bigg|_{\VEC{y}=\VEC{1}}=\frac{rm}{1-(1-r)m}.
	\end{equation*}
	Finally, by~\cite[Theorem V.3.2]{AthreyaNey}, the process $(\VEC{M}_k, k \in \N_0)$
	is \emph{supercritical}, \emph{critical}, or \emph{subcritical} according as
	$\rho > 1$, $\rho = 1$, or $\rho < 1$, respectively.
	A straightforward calculation then yields the statement of the proposition.
\end{proof}

\section{The $Q$-process associated with the mutant chain}\label{sec: backmulticond}

The fundamentals of multitype BGW processes conditioned on non-extinction in the
distant future were originally developed by \cite{nakagawa1978} and later
extended by \cite{penisson2014}. Therefore, the $Q$-process associated with
$(\VEC{M}_k,\, k\in\N_0)$ is well defined, and we will show that its distribution can be
described in terms of $(M_k^\uparrow,\, k\in \N_0)$ and a multinomial colouring, as we
point out in Figure~\ref{fig:conmDiagram}.
As a first step, we determine the extinction probability vector of the mutant chain in
the finite-alleles case in terms of the extinction probability of the mutant-count
process.

\subsection{Extinction of the mutant chain}

Recall that $\tau_\VEC{M}$ denotes the extinction time of the $d$-type mutant chain
$(\VEC{M}_k, k \in \N_0)$, defined by~\eqref{eq: extime}.
By Lemma~\ref{lem: eqMand|VECM|}, $\tau_\VEC{M}$ is equal in law to $\tau$, the
extinction time of $(M_k, k \in \N_0)$.
Henceforth, we assume hypothesis~\eqref{hyp:timext}.
Let
\begin{equation}\label{eq: exttotmut}
	\VEC{q}(i)\coloneqq \P_{i}(\tau_\VEC{M}<\infty) = \P_{i}\big(|\VEC{M}_k|=0 \text{ for some } k \in \N\big), \qquad i \in [d]
\end{equation}
be the \emph{extinction probability vector}.
Hypothesis~\eqref{hyp:timext} is equivalent to the assumption that $\VEC{q}$ satisfies
\begin{equation}\label{eq:extstricpost}
	\VEC{q}=(\VEC{q}(1), \dots,\VEC{q}(d) )> \VEC{0}.
\end{equation}
In the (sub)critical case $(m \le 1)$ we have $\VEC{q} = \VEC{1}$, whereas in the
supercritical case $(m > 1)$, $\VEC{q} \boldsymbol{<} \VEC{1}$;
see~\cite[Theorem~V.3.2(ii)]{AthreyaNey}.

Since the law of the mutant-count process is the same when starting from an individual
of any type, it is natural that the extinction probability of
$(|\VEC{M}_k|, k \in \N_0)$ coincides with that of a single-type BGW process with
offspring distribution $\mu^+$, corresponding to the underlying branching process of
the mother-dependent mutation model in which mutations occur.
The following result formalises this intuition.

\begin{proposition} \label{prop: q}
	Let $G(s)\coloneqq\sum_{n \in \N_0}\mu^+(n)s^n$, $s\in[0,1]$, be the p.g.f. of
	$\mu^+$, and let $q$ denote the extinction probability of the BGW process
	$(Z^{(+)}_k, k \in \N_0)$, that is, the smallest solution in $[0,1]$ of
	\begin{equation} \label{eq: pgfmu}
		G(s)=s.
	\end{equation}
	Then:
	\begin{itemize}
		\item[(i)] the $d$-type mutant chain $( \VEC{M}_k,\, k \in\N_0 )$ has extinction
		probability vector $\VEC{q} = q\VEC{1}$;
		\item[(ii)] the mutant-count chain $(|\VEC{M}_k|, k \in \N_0 )$ has extinction
		probability $q$.
	\end{itemize}
\end{proposition}

\begin{proof}
	We begin with part~(i).
	By~\cite[Section V.3 Theorem 2 (iii)]{AthreyaNey}, the extinction probability vector $\VEC{q}$ satisfies $\VEC{f}(\VEC{q})=\VEC{q}$. Let $\VEC{s}\in[0,1]^d$ be an arbitrary point satisfying $\VEC{f}(\VEC{s})=\VEC{s}$. Then, using~\eqref{eq: pgfM_i}
	in the first line and~\eqref{eq: pgfTM_i} with $x=1$ in the second,
	\begin{equation*}
		\begin{split}
			\VEC{s}(i)&= f^{(i)}(\VEC{s})
			=\varphi^{(i)}(1,\VEC{s})
			= g_i\big(\varphi^{(i)}(1,\VEC{s}),\overline{\VEC{s}}^{\,i}\big)
			= g^{(i)}(\VEC{s})
			= \sum_{\VEC{v} \in \N_0^d} \VEC{s}^{\VEC{v}}\,\boldsymbol{\mu}_i(\VEC{v})\\
			&= \sum_{k \in \N_0}\mu^+(k)\sum_{\VEC{v}\in\N_0^d:\, |\VEC{v}|=k} \binom{k}{\VEC{v}(1),\dots,\VEC{v}(d)} [(1-r)\VEC{s}(i)]^{\VEC{v}(i)} \prod_{j \neq i}\big(\VEC{s}(j)r \VEC{P}_{ij}\big)^{\VEC{v}(j)}\\
            &= \sum_{k \in \N_0}\mu^+(k){\overline p_i}^{k}\sum_{\VEC{v}\in\N_0^d:\, |\VEC{v}|=k} \binom{k}{\VEC{v}(1),\dots,\VEC{v}(d)} \left[\frac{(1-r)\VEC{s}(i)}{\overline p_i}\right]^{\VEC{v}(i)} \prod_{j \neq i}\left[\frac{\VEC{s}(j)r \VEC{P}_{ij}}{\overline p_i}\right]^{\VEC{v}(j)}\\
			&= \sum_{k \in \N_0}\mu^+(k)\,\overline{p}_i^{\,k}
			= G(\overline p_i),
		\end{split}
	\end{equation*}
	where we used that $\varphi^{(i)}(1,\VEC{s})=f^{(i)}(\VEC{s})=\VEC{s}(i)$ in the third
	equality, and where
	\begin{equation}\label{eq: renorm}
		\overline{p}_i\coloneqq (1-r)\VEC{s}(i) + r\sum_{j \neq i} \VEC{P}_{ij}\,\VEC{s}(j)
	\end{equation}
	is the normalizing constant ensuring that the inner sum is the probability mass
	function of a multinomial distribution. Thus any fixed point of $\VEC{f}$ satisfies
	\begin{equation}\label{eq: qip}
		\VEC{s}(i)= G(\overline p_i), \qquad \forall\, i\in [d].
	\end{equation}

	Moreover, set $a\coloneqq\max_i \VEC{s}(i)$
	and $b\coloneqq\min_i \VEC{s}(i)$. By~\eqref{eq: renorm}, $\overline p_i$ is a convex combination of coordinates of $\VEC{s}$, so that $b\le \overline p_i \le a$ for every $i$. Since $G$ is non-decreasing, evaluating \eqref{eq: qip} at an index attaining the maximum and at one attaining the minimum
	gives
	\[
		a= G(\overline p_{max}) \le  G(a) \qquad\text{and}\qquad G(b)\le G(\overline p_{min}) = b.
	\]
	We know that the function $G$ is convex with $G(1)=1$, and $q$ is its smallest fixed point; hence
	$G(s)>s$ for $s\in[0,q)$ and $G(s)<s$ for $s\in(q,1)$.
	The inequality $G(b)\le b $ therefore forces $q\le b $, and the inequality $a \le G(b)$ forces $a\le b$. In both cases $q\le b\le a\le q$, that is, if $\VEC{s}$ is a fixed point, has to be $\VEC{q}=q\VEC{1}$.

	Conversely, $q\VEC{1}$ is indeed a fixed point of $\VEC{f}$: substituting
	$\VEC{s}=q\VEC{1}$ in~\eqref{eq: renorm} gives $\overline p_i=q$, and then
	\eqref{eq: qip} reads $q=G(q)$, which holds by definition of $q$.

	To prove item~(ii) it is enough to note that, by Lemma~\ref{prop: mutbackfor},
	\[
		h(q) = f^{(i)}(q\VEC{1})=q \qquad \text{ for all } i\in [d],
	\]
	so $q$ solves $h(s)=s$. Minimality of $q$ among the solutions of $h(s)=s$ follows
	from part~(i) and Lemma~\ref{lem: eqMand|VECM|}, since $h$ is the offspring
	generating function of the BGW process $(|\VEC{M}_k|, k \in \N_0 )$.
\end{proof}

\begin{remark}\label{re:Hcharacterisation}
	Proposition~\ref{prop: q} shows that hypothesis~\eqref{hyp:timext} amounts to $q>0$,
	which in turn holds if and only if $\mu^+(0)>0$. Likewise, \eqref{hyp:timextsubc}
	holds if and only if $m\le 1$, and~\eqref{hyp:timextsupc} if and only if $m>1$ and
	$\mu^+(0)>0$.
\end{remark}

\subsection{The $Q$-process of the mutant chain}

Under~\eqref{hyp:timext}, the law of the $Q$-process $(\VEC{M}_k^\uparrow, k\in \N_0)$
is defined by
\begin{equation}\label{eq: defQ}
	\P_{\VEC{y}_0}( \VEC{M}_{k_1}^\uparrow= \VEC{y}_1, \dots, \VEC{M}_{k_\ell}^\uparrow= \VEC{y}_\ell )
	\coloneqq
	\lim_{n \rightarrow \infty} \P_{\VEC{y}_0}( \VEC{M}_{k_1}= \VEC{y}_1, \dots, \VEC{M}_{k_\ell}= \VEC{y}_\ell \mid k_\ell + n < \tau_{\VEC{M}} <\infty ),
\end{equation}
for all $\VEC{y}_0\in S^\ast$, $k_1< \dots< k_\ell\in\N$, and
$\VEC{y}_1, \dots, \VEC{y}_\ell\in S^\ast$.
This process is sometimes referred to as the process conditioned on very late
extinction, or the process conditioned on extinction in the distant future. Under
hypothesis~\eqref{hyp:timextsubc} extinction is almost sure, so that the conditioning
event in~\eqref{eq: defQ} reduces to $\{\tau_{\VEC{M}}>k_\ell+n\}$.

A single object controls both regimes. Consider the generating function
$\overline{\VEC{f}} \coloneqq (\overline f^{(1)}, \dots, \overline f^{(d)})$, defined coordinatewise by
\begin{equation}\label{eq barf}
	\overline{f}^{(i)}(\VEC{y})\coloneqq
	\frac{f^{(i)}( \VEC{q}\VEC{y} )}{ f^{(i)}( \VEC{q})}= \frac{ f^{(i)}( q\VEC{y} )}{q}, \qquad i\in[d],
\end{equation}
where the last equality uses Proposition~\ref{prop: q}(i) and $q>0$.
The function $\overline{\VEC{f}}$ is the generating function of the process
$(\VEC{M}_k, k \in \N_0)$ conditioned on extinction, which we denote by
$(\overline{\VEC{M}}_k, k \in \N_0)$; it is a (sub)critical $d$-type BGW process. Under
hypothesis~\eqref{hyp:timextsubc} we have $q=1$ and $\overline{\VEC{f}}=\VEC{f}$.

\begin{proposition}\label{prop:barA}
	Assume~\eqref{hyp:timext} and let $\overline{\VEC{A}}$ be the mean matrix of
	$(\overline{\VEC{M}}_k,\, k \in \N_0)$. Then
	\[
		\overline{\VEC{A}}_{ii}= 0,
		\qquad \text{and} \qquad
		\overline{\VEC{A}}_{ij}= \E_{i}\big[ q^{ |\VEC{M}_1|-1} \VEC{M}_1(j) \big], \qquad i \neq j.
	\]
	Moreover, $\overline{\VEC{A}}=h'(q)\VEC{P}$, its maximal eigenvalue is
	\begin{equation}\label{eq:barrho}
		\overline{\rho}\coloneqq\E_{i}\big[q^{|\VEC{M}_1|-1}|\VEC{M}_1|\big] =h'(q) =\sum_{j\neq i}\frac{\partial f^{(i)}(\VEC{y})}{\partial\VEC{y}(j)}\bigg|_{\VEC{y}=q\VEC{1}},
	\end{equation}
	and the associated right and left eigenvectors are $\VEC{1}$ and $\VEC{u}$,
	respectively. Finally $\overline\rho\le 1$, with equality if and only if $\rho=1$.
\end{proposition}

\begin{proof}
	Take $i \in [d]$. By definition of the $d$-type mutant process,
	$\overline{\VEC{A}}_{ii}= \E_{i}[\overline{\VEC{M}}_1(i)]= 0$ for all $i\in [d]$. For $j\neq i$,
	using the generating function $\overline{\VEC{f}}$ defined in~\eqref{eq barf},
	\begin{align*}
		\overline{\VEC{A}}_{ij} & = \E_{i}\big[\overline{\VEC{M}}_1(j) \big]
		= \frac{\partial }{\partial \VEC{s}(j)} \overline{f}^{(i)}(\VEC{s}) \bigg|_{\VEC{s}=\VEC{1}}
		= \frac{\partial }{\partial \VEC{s}(j)} \frac{1}{q} f^{(i)}(q\VEC{s}) \bigg|_{\VEC{s}=\VEC{1}} \\
		& = \frac{1}{q}
		\E_{i}\left[\VEC{M}_1(j) q^{|\VEC{M}_1|} \VEC{s}(j)^{\VEC{M}_1(j)-1} \prod_{k\neq i,j}\VEC{s}(k)^{\VEC{M}_1(k)} \right] \Bigg|_{\VEC{s}=\VEC{1}}
		= \E_{i}\big[ q^{ |\VEC{M}_1|-1} \VEC{M}_1(j)\big].
	\end{align*}
	Taking conditional expectation with respect to $|\VEC{M}_1|$ and using
	Lemma~\ref{lem: eqMand|VECM|} yields
	\[
		\overline{\VEC{A}}_{ij}= \VEC{P}_{ij}\,\E_{i}\big[q^{|\VEC{M}_1|-1}|\VEC{M}_1|\big]=\VEC{P}_{ij}\,h'(q) \qquad \Longleftrightarrow \qquad \overline{\VEC{A}} = h'(q) \VEC{P}
	\]
	where $h$, defined in~\eqref{eq: pgf |m|}, is the generating function of the
	mutant-count process. It follows that the maximal eigenvalue of 	
    $\overline{\VEC{A}}$ is $h'(q)$, with associated right eigenvector $\VEC{1}$ and left eigenvector $\VEC{u}$.
	Finally, $h$ is convex on $[0,1]$ with $h(q)=q$ and $h(1)=1$. If $q<1$ this forces
	$h'(q)<1$, while if $q=1$ then $h'(1)=\rho\le 1$ by
	Proposition~\ref{prop: meanmxrho}, with equality exactly in the critical case.
\end{proof}

We can now state the law of the $Q$-process of the mutant chain in a form that covers
both regimes. Its proof is deferred to Appendix~\ref{app:proof4}; it is a particular
case of~\cite[Equation 2.4]{penisson2014} and is based on the ideas of
~\cite[Section~3]{nakagawa1978}, whose paper is not easily accessible.

\begin{theorem}\label{thm: Qprocess}
	Assume~\eqref{hyp:timext} and let $\VEC{y}_0\in S^\ast$. The limit
	in~\eqref{eq: defQ} exists and defines the law of the $Q$-process
	$(\VEC{M}^\uparrow_{k}, k\in \N_0)$, given by
	\begin{equation} \label{eq:Qlim}
		\P_{\VEC{y}_0}( \VEC{M}^\uparrow_{k_1}= \VEC{y}_1, \dots, \VEC{M}^\uparrow_{k_\ell}= \VEC{y}_\ell )
		= \P_{\VEC{y}_0}( \VEC{M}_{k_1}= \VEC{y}_1, \dots, \VEC{M}_{k_\ell}= \VEC{y}_\ell ) \frac{q^{|\VEC{y}_\ell|-|\VEC{y}_0|}}{\big(h^{\prime}(q)\big)^{k_\ell}} \frac{|\VEC{y}_\ell |}{|\VEC{y}_0|},
	\end{equation}
	for any $k_1<\dots<k_\ell$ and $\VEC{y}_1,\dots, \VEC{y}_\ell \in S^\ast$,
	where $q$ is the extinction probability of $( \VEC{M}_k,\, k \in\N_0 )$ and $h$ is the
	generating function of $| \VEC{M}_1 |$.
	Moreover, $(\VEC{M}^\uparrow_{k}, k\in \N_0)$ is a Markov chain with transition
	probabilities
	\begin{equation}\label{eq: Qtransitions}
		Q_{(\VEC{x}, \VEC{y})} = \frac{ q^{|\VEC{y}|-|\VEC{x}|}}{h'(q)} \frac{ | \VEC{y} | }{ | \VEC{x} |}P_{(\VEC{x}, \VEC{y})}, \qquad \VEC{x},\VEC{y} \in S^\ast,
	\end{equation}
	where $P_{(\VEC{x},\VEC{y})}$ are the transition probabilities from $\VEC{x}$ to
	$\VEC{y}$ of the process $( \VEC{M}_k,\, k \in\N_0 )$.

	In particular, under~\eqref{hyp:timextsubc} one has $q=1$ and $h'(q)=\rho$, so
	that~\eqref{eq:Qlim} reads
	\begin{equation} \label{eq:Qlimsub}
		\P_{\VEC{y}_0}( \VEC{M}^\uparrow_{k_1}= \VEC{y}_1, \dots, \VEC{M}^\uparrow_{k_\ell}= \VEC{y}_\ell )
		= \P_{\VEC{y}_0}( \VEC{M}_{k_1}= \VEC{y}_1, \dots, \VEC{M}_{k_\ell}= \VEC{y}_\ell ) \left(\frac{1-(1-r)m}{rm}\right)^{k_\ell} \frac{|\VEC{y}_\ell |}{|\VEC{y}_0|}.
	\end{equation}
\end{theorem}

\subsection{The mutant-count $Q$-process and the multinomial colouring}

We now consider the mutant-count chain $(|\VEC{M}^\uparrow_k|,\, k \in \N_0)$
associated with $(\VEC{M}^\uparrow_k,\, k \in \N_0)$.
By Lemma~\ref{lem: eqMand|VECM|}, it is natural to expect its law to be related to that
of $(M^\uparrow_k,\, k \in \N_0)$, the $Q$-process of the infinite-alleles mutant chain
$(M_k,\, k \in \N_0)$, which was studied in~\cite{BlancasRivero} under mild assumptions
excluding degenerate cases.

\begin{corollary}\label{cor: Conmuta}
	Assume~\eqref{hyp:timext}. The $Q$-process associated with
	$(|\VEC{M}_k|,\, k\in\N_0)$ coincides with $(|\VEC{M}_k^\uparrow|, k\in \N_0)$.
	Indeed, for each $i \in [d]$, under the measure $\P_{i}$,
	\begin{equation} \label{eq:countQequality}
		(|\VEC{M}_k^\uparrow|, k\in \N_0) \overset{\mathcal{L}}{=} (M^{\uparrow}_k, k\in \N_0).
	\end{equation}
	Consequently, $( |\VEC{M}^\uparrow_k|-1,\ k\in\N_0 )$ is a single-type BGW process
	with immigration $\big[\overline h,\ \overline h^\prime/\overline h^\prime(1) \big]$, where
	$\overline h(s)\coloneqq h(qs)/q$ is the offspring p.g.f. and the second entry is the
	immigration p.g.f. Under~\eqref{hyp:timextsubc} this reads
	$\big[h,\ h^\prime/\rho \big]$. 
\end{corollary}
\begin{proof}
	For all $n \in \N$ and $i\in [d]$, taking $\ell=1$, $k_1=1$ and
	$\VEC{y}_0=\VEC{e}_i$ in~\eqref{eq:Qlim} gives
	\begin{equation}\label{eq: exp1}
		\P_{i}(\VEC{M}_1^{\uparrow}=\VEC{v})=\P_{i}(\VEC{M}_1=\VEC{v})\,\frac{q^{|\VEC{v}|-1}|\VEC{v}|}{h'(q)},
		\qquad \VEC{v}\in S^\ast,
	\end{equation}
	and therefore
	\begin{equation}\label{eq: sizebiasedgral}
		\begin{split}
			\P_{i}(|\VEC{M}_1^{\uparrow}|=n)&= \sum_{\VEC{v}:\, |\VEC{v}|=n }\P_{i}(\VEC{M}_1^{\uparrow}=\VEC{v})
			=\frac{nq^{n-1}}{h'(q)}\sum_{\VEC{v}:\, |\VEC{v}|=n }\P_{i}(\VEC{M}_1=\VEC{v})
			= \frac{nq^{n-1}}{h'(q)}\,\P_{i}(|\VEC{M}_1|=n).
		\end{split}
	\end{equation}

	According to~\cite[Section I.14]{AthreyaNey}, the right-hand side
	of~\eqref{eq: sizebiasedgral} is precisely the one-step law of the $Q$-process of a
	single-type BGW process whose offspring generating function is $h$ and whose
	extinction probability is $q$; note that it is a probability distribution because
	$\sum_n nq^{n-1}\P_i(|\VEC{M}_1|=n)=h'(q)$.
	Since $(|\VEC{M}_k|,\, k\in\N_0)$ is a single-type BGW process, the branching property
	implies that $(|\VEC{M}_k^\uparrow|, k\in \N_0)$ is the $Q$-process associated with
	$(|\VEC{M}_k|,\, k\in\N_0)$, and Lemma~\ref{lem: eqMand|VECM|}
	gives~\eqref{eq:countQequality}.
	The last assertion is~\cite[Corollary~3.2]{BlancasRivero} in the (sub)critical case.
	In general, it follows from the classical description of the $Q$-process of a
	single-type BGW process conditioned on extinction, see~\cite[Section I.14]{AthreyaNey}
	and~\cite{LyonsPemantlePeres}, applied to the subcritical offspring generating
	function $\overline h$, for which $\overline h'(1)=h'(q)=\overline\rho<1$ when $q<1$.
\end{proof}

\begin{proposition} \label{prop:colorChain}
	Assume~\eqref{hyp:timext}. Conditionally on the total offspring size, the offspring
	vector of an individual of type $i$ in the $Q$-process is multinomially distributed:
	under $\P_i$,
	\begin{equation}\label{eq:condtomcpmatixsub}
		\VEC{M}_1^\uparrow\big|_{\{\VEC{M}^\uparrow_0=\VEC{e}_i,\,|\VEC{M}_1^\uparrow|=n\}}\sim \mathcal{M}(n,\VEC{P}_{i\cdot}).
	\end{equation}
	Moreover, $(\VEC{M}^\uparrow_k,\, k \in \N_0)$ is a Markov chain with transition
	probabilities
	\begin{multline*}
			\P_{\VEC{y}_0} \big( \VEC{M}^\uparrow_{k+1}= \VEC{y}_{k+1} \Big|\VEC{M}^\uparrow_{k}= \VEC{y}_k \big) = \\\frac{q^{|\VEC{y}_{k+1}|-|\VEC{y}_k|}}{h'(q)} \frac{|\VEC{y}_{k+1} |}{|\VEC{y}_k|}\,
			\P_{\VEC{y}_k} \big( |\VEC{M}_{1}|= |\VEC{y}_{k+1}|\big)
			\P_{\VEC{y}_k} \big( \VEC{M}_{1}= \VEC{y}_{k+1} \Big| |\VEC{M}_{1}|= |\VEC{y}_{k+1}| \big).
	\end{multline*}
	In other words, the transition probabilities of the $Q$-process
	$(\VEC{M}^\uparrow_k,\, k \in \N_0)$ describe a two-step procedure:
	first, we sample the total offspring count $|\VEC{M}^\uparrow_{k+1}|$ from the BGW
	process with immigration of Corollary~\ref{cor: Conmuta}; second, the offspring types
	are assigned independently and multinomially, using the row of $\VEC{P}$ indexed by
	the parent's type.
\end{proposition}

\begin{proof}
	For all $n \in \N$ and $i\in [d]$, taking $\ell=1$, $k_1=1$ and
	$\VEC{y}_0=\ell\VEC{e}_i$ in~\eqref{eq:Qlim} gives
	\begin{equation*}
		\P_{\ell\VEC{e}_i}(\VEC{M}_1^{\uparrow}=\VEC{v})=\P_{\ell\VEC{e}_i}(\VEC{M}_1=\VEC{v})\,\frac{q^{|\VEC{v}|-\ell}}{h'(q)}\frac{|\VEC{v}|}{\ell},
		\qquad \VEC{v}\in S^\ast,
	\end{equation*}
	and therefore
	\begin{equation}\label{eq: sizebiasedgral1}
		\begin{split}
			\P_{\ell\VEC{e}_i}(|\VEC{M}_1^{\uparrow}|=n)&= \sum_{\VEC{v}:\, |\VEC{v}|=n }\P_{\ell\VEC{e}_i}(\VEC{M}_1^{\uparrow}=\VEC{v})
			=\frac{q^{n-\ell}}{h'(q)}\frac{n}{\ell}\sum_{\VEC{v}:\, |\VEC{v}|=n }\P_{\ell\VEC{e}_i}(\VEC{M}_1=\VEC{v})\\
			&= \frac{q^{n-\ell}}{h'(q)}\frac{n}{\ell}\sum_{\VEC{v}:\, |\VEC{v}|=n }\P_{\ell}(|\VEC{M}_1|=n).
		\end{split}
	\end{equation}
According to~\cite[Section I.14]{AthreyaNey}, the right-hand side
	of~\eqref{eq: sizebiasedgral1} is the one-step law of the $Q$-process of a
	single-type BGW process whose offspring generating function is $h$ and whose
	extinction probability is $q$. Note that it is a probability distribution because
	$\sum_n nq^{n-1}\P_i(|\VEC{M}_1|=n)=h'(q)$.
	Since $(|\VEC{M}_k|,\, k\in\N_0)$ is a single-type BGW process, the branching property
	implies that $(|\VEC{M}_k^\uparrow|, k\in \N_0)$ is the $Q$-process associated with
	$(|\VEC{M}_k|,\, k\in\N_0)$, and Lemma~\ref{lem: eqMand|VECM|}
	gives~\eqref{eq:countQequality}.

	The last assertion is~\cite[Corollary~3.2]{BlancasRivero} in the (sub)critical case;
	in general, it follows from the classical description of the $Q$-process of a
	single-type BGW process conditioned on extinction, see~\cite[Section I.14]{AthreyaNey}. Indeed, by \cite{kawazu1971branching}, we know that a BGW process with immigration corresponding to
$[f(s), g(s)]$ is a time-discrete Markov chain $(Z_k)$ with p.g.f. given by 
\begin{equation*}
    \E_1\big[s^{Z_1}\big]= \E\big[s^{X_1 + Y_{1}}\big]=f(s)g(s),
\end{equation*}
where $X_1$ corresponds to the offspring of the original individual and $Y_1$ is the immigration in the first generation, so $f(\cdot)$ is the offspring p.g.f. and $g(\cdot)$ is the immigration p.g.f.
In particular, if we take $Z_0\coloneqq|\VEC{M}^\uparrow_0|-1=1$, this implies that $|\VEC{M}^\uparrow_0|=2$, and we can set $\VEC{M}^\uparrow_0=2\VEC{e}_i$ without loss of generality, so by \eqref{eq: sizebiasedgral} we have
\begin{equation*}
\begin{split}
    \E_1\big[s^{Z_1}\big]&=
    \E_{2\VEC{e}_i}\big[s^{|\VEC{M}^\uparrow_1|-1}\big]=
    \sum\limits_{n=1}^{\infty}s^{n-1}\frac{q^{n-2}}{h'(q)}\frac{n}{2}\,\P_{2\VEC{e}_i}(|\VEC{M}_1|=n)= 
    \frac{1}{2qh'(q)}\sum\limits_{n=1}^{\infty}n(qs)^{n-1}\P_{2\VEC{e}_i}(|\VEC{M}_1|=n)\\
    &=\frac{1}{2qh'(q)}2h(qs)h'(qs) = \frac{h(qs)}{q}\frac{h'(qs)}{h'(q)}= \frac{\overline{h}(s)}{q}\frac{\overline{h}'(s)}{\overline{h}'(1)},
\end{split}
\end{equation*}
where to obtain the fourth equality, note that, as $(|\VEC{M}_k|)$ is a BGW process, we have 
\begin{equation*}
\sum\limits_{n=0}^{\infty}(qs)^{n}\P_{2\VEC{e}_i}(|\VEC{M}_1|=n) =h^2(qs),
\end{equation*}
and we differentiate both sides with respect to $s$. 
\end{proof}
\section{The $Q$-process associated with the clone-mutant chain}
\label{sec:clone_mutant_nonExtinction}

Thanks to the results established in the previous section, we are now in a position to
define the clone-mutant chain $\big((\VEC{T}_k,\VEC{M}_{k+1}),\, k \in \N_0 \big)$
conditioned on not being absorbed in the distant future. As in
Theorem~\ref{thm: Qprocess}, a single statement covers both regimes; under
\eqref{hyp:timextsubc} it should be read with $q=1$ and
$h'(q)=\rho=\tfrac{rm}{1-(1-r)m}$.

Thanks to Remark \ref{re: sameextime}, it is clear that the conditional law of the process $\bigl( (\VEC{T}_\ell, \VEC{M}_{\ell+1}) : \ell \in \{0, 1, \ldots, k-1\} \bigr)$ under  $\P_{\VEC{y}_0}\!\big( \,\cdot\, \mid k+n < \tau_\VEC{M} < \infty \big)$ is the same as that the process under \\$\P_{\VEC{y}_0}\!\big( \,\cdot\, \mid k+n < \tau_{(\VEC{T},\VEC{M})} < \infty \big)$, and therefore we have the following result.

\begin{theorem}\label{thm:condnonext_clonemut}
	Assume hypothesis~\eqref{hyp:timext}.
	Let $\VEC{y}_0 \in S^\ast$ and fix $k \in \N$. Then the conditional law of the
	process
	\[
		\bigl( (\VEC{T}_\ell, \VEC{M}_{\ell+1}) : \ell \in \{0, 1, \ldots, k-1\} \bigr)
		\text{ under } \P_{\VEC{y}_0}\!\big( \,\cdot\, \mid k+n < \tau_\VEC{M} < \infty \big)
	\]
	converges as $n \to \infty$ to $\P^{\uparrow}_{\VEC{y}_0}$, 
	given by
	\begin{align*}
		\P^{\uparrow}_{\VEC{y}_0}\bigl(&\VEC{T}_0 = \VEC{x}_0, \VEC{M}_1 = \VEC{y}_1, \ldots, \VEC{T}_{k-1} = \VEC{x}_{k-1}, \VEC{M}_k = \VEC{y}_k \bigr) \\
		&= \P_{\VEC{y}_0}\bigl(\VEC{T}_0 = \VEC{x}_0, \VEC{M}_1 = \VEC{y}_1, \ldots, \VEC{T}_{k-1} = \VEC{x}_{k-1}, \VEC{M}_k = \VEC{y}_k\bigr)\,
		\frac{q^{|\VEC{y}_k| - |\VEC{y}_0|}}{h'(q)^{k}}\,
		\frac{|\VEC{y}_k|}{|\VEC{y}_0|},
	\end{align*}
	for all $\VEC{x}_0, \ldots, \VEC{x}_{k-1} \in \N_0^d$ and
	$\VEC{y}_1, \ldots, \VEC{y}_k \in S^\ast$.
    Moreover, this is the law of the $Q$-process\\
	$\bigl( (\VEC{T}^{\uparrow}_k, \VEC{M}^{\uparrow}_{k+1}),\, k \in \N_0 \bigr)$, associated with
	$\bigl( (\VEC{T}_k, \VEC{M}_{k+1}),\, k \in \N_0 \bigr)$, with transition
	probabilities
	\begin{multline} \label{eq: qtrans_clonemutd}
		\P^{\uparrow}_{\VEC{y}_0}\!\big( \VEC{T}_k = \VEC{x}_k, \VEC{M}_{k+1} = \VEC{y}_{k+1} \,\Big|\, \VEC{T}_{k-1} = \VEC{x}_{k-1}, \VEC{M}_k = \VEC{y}_k \big)
		= \\ \frac{q^{|\VEC{y}_{k+1}| - |\VEC{y}_k|}}{h'(q)}\,
		\frac{|\VEC{y}_{k+1}|}{|\VEC{y}_k|}\,
		\P_{\VEC{y}_k}\!\big( \VEC{T}_0 = \VEC{x}_k, \VEC{M}_1 = \VEC{y}_{k+1} \big).
	\end{multline}

\end{theorem}

\begin{proof}
	For every $k\in \N$, consider the event
	\[
		E_k\coloneqq\{ \VEC{T}_0 =\VEC{x}_0, \VEC{M}_{1}=\VEC{y}_1, \dots , \VEC{T}_{k-1}=\VEC{x}_{k-1},\VEC{M}_{k}=\VEC{y}_{k} \}.
	\]
	Write $P^{n}_{(\VEC{x}, \VEC{y})}$ for the $n$-step transition probabilities of
	$( \VEC{M}_k,\, k\in\N_0 )$ and recall from~\eqref{eq: exttotmut} and
	Proposition~\ref{prop: q} that $\VEC{q}^{\VEC{z}}=q^{|\VEC{z}|}$ is the probability
	that the mutant chain started from $\VEC{z}$ eventually dies out. By the Markov and
	branching properties,
	\[
		\P_{\VEC{y}_0}(k+n<\tau_\VEC{M}<\infty)
		= \E_{\VEC{y}_0}\bigl[ \ind_{ \{\VEC{M}_{k+n} \neq \VEC{0} \}} \VEC{q}^{\VEC{M}_{k+n}}\bigr]
		=\sum_{\VEC{z}\in S^\ast} P^{k+n}_{(\VEC{y}_0, \VEC{z})}\, q^{|\VEC{z}|}.
	\]
	Similarly, since $\VEC{M}_k=\VEC{y}_k$ on $E_k$ and $E_k$ is measurable with respect
	to the first $k$ allelic generations,
	\[
		\P_{\VEC{y}_0}(E_k,\, k+n<\tau_\VEC{M}<\infty)
		= \E_{\VEC{y}_0}\Bigl[ \ind_{E_k}\,
		\E_{\VEC{M}_k} \bigl[ \ind_{ \{\VEC{M}_{n} \neq \VEC{0} \}} \VEC{q}^{\VEC{M}_{n}} \bigr] \Bigr]
		= \P_{\VEC{y}_0}(E_k) \sum_{\VEC{z}\in S^\ast} P^{n}_{(\VEC{y}_k, \VEC{z})}\, q^{|\VEC{z}|}.
	\]
	Therefore,
	\[
		\P_{\VEC{y}_0}(E_k \mid k+n<\tau_\VEC{M}<\infty)=\P_{\VEC{y}_0}(E_k \mid k+n<\tau_{(\VEC{T},\VEC{M})}<\infty)
		= \P_{\VEC{y}_0}(E_k)\,
		\frac{\sum_{\VEC{z}\in S^\ast} P^{n}_{( \VEC{y}_{k}, \VEC{z})} q^{|\VEC{z}|}}{\sum_{\VEC{z}\in S^\ast} P^{k+n}_{(\VEC{y}_0, \VEC{z})} q^{|\VEC{z}|}},
	\]
	and by using~\eqref{eq: limratnstepgmM} we obtain
	\[
		\P^{\uparrow}_{\VEC{y}_0}(E_k)= \lim_{n\rightarrow \infty} \P_{\VEC{y}_0}(E_k \mid k+n<\tau_\VEC{M}<\infty)
		= \P_{\VEC{y}_0}(E_k)\, \frac{1}{h'(q)^{k}}\, \frac{q^{|\VEC{y}_k|}|\VEC{y}_k |}{q^{|\VEC{y}_0|}|\VEC{y}_0|}.
	\]

	Using this expression, we compute
	\begin{equation*}
		\begin{split}
			\P^{\uparrow}_{\VEC{y}_0}&\big( \VEC{T}_{k} = \VEC{x}_{k}, \VEC{M}_{k+1}=\VEC{y}_{k+1} \mid \VEC{T}_{k-1} =\VEC{x}_{k-1}, \VEC{M}_{k}=\VEC{y}_{k} \big)\\
			&= \frac{\P^{\uparrow}_{\VEC{y}_0}\big( \VEC{T}_{k} = \VEC{x}_{k}, \VEC{M}_{k+1}=\VEC{y}_{k+1}, \VEC{T}_{k-1} =\VEC{x}_{k-1}, \VEC{M}_{k}=\VEC{y}_{k} \big)}{\P^{\uparrow}_{\VEC{y}_0}\big( \VEC{T}_{k-1} =\VEC{x}_{k-1}, \VEC{M}_{k}=\VEC{y}_{k} \big)}\\
			&=\frac{\frac{q^{|\VEC{y}_{k+1}|-|\VEC{y}_0|}}{h'(q)^{k+1}}\frac{|\VEC{y}_{k+1}|}{|\VEC{y}_0|}\P_{\VEC{y}_0}\big( \VEC{T}_{k} = \VEC{x}_{k}, \VEC{M}_{k+1}=\VEC{y}_{k+1}, \VEC{T}_{k-1} =\VEC{x}_{k-1}, \VEC{M}_{k}=\VEC{y}_{k} \big)}{\frac{q^{|\VEC{y}_{k}|-|\VEC{y}_0|}}{h'(q)^{k}}\frac{|\VEC{y}_{k}|}{|\VEC{y}_0|}\P_{\VEC{y}_0}\big( \VEC{T}_{k-1} =\VEC{x}_{k-1}, \VEC{M}_{k}=\VEC{y}_{k} \big)}\\
			&=\frac{q^{|\VEC{y}_{k+1}|- |\VEC{y}_k|}}{h'(q)} \frac{ | \VEC{y}_{k+1} | }{ | \VEC{y}_k |}\P_{\VEC{y}_0}\big( \VEC{T}_{k} = \VEC{x}_{k}, \VEC{M}_{k+1}=\VEC{y}_{k+1} \Big| \VEC{T}_{k-1} =\VEC{x}_{k-1}, \VEC{M}_{k}=\VEC{y}_{k} \big)\\
			&=\frac{q^{|\VEC{y}_{k+1}|- |\VEC{y}_k|}}{h'(q)} \frac{ | \VEC{y}_{k+1} | }{ | \VEC{y}_k |}\P_{\VEC{y}_k}\big( \VEC{T}_{0} =\VEC{x}_{k}, \VEC{M}_{1}=\VEC{y}_{k+1} \big),
		\end{split}
	\end{equation*}
	where the last equality holds by Proposition~\ref{prop:GTyMMarkov}.
	This identifies the transition probabilities in~\eqref{eq: qtrans_clonemutd} and
	completes the proof.
\end{proof}

\section{Spine decomposition for the mutant chain} \label{sec: spine}

Throughout this section we assume~\eqref{hyp:size}, so that
$\rho=\E_i[|\VEC{M}_1|]<\infty$.
Since the $d$-type mutant chain $(\VEC{M}_k,\, k \in \N_0)$ is a $d$-type BGW process,
the spine decomposition for multitype branching processes developed
in~\cite{BigKyp,GeorgiiBaake} applies, and we can give an explicit construction of the
associated size-biased tree using Lemma~\ref{lem: eqMand|VECM|} and
Proposition~\ref{prop: meanmxrho}. We follow the presentation
of~\cite[Section 3.3]{LuiRen}.

The genealogy of $(\VEC{M}_k,\, k \in \N_0)$ is encoded as a random tree $\mathcal{T}$
in which each vertex carries a type, or \emph{mark}, in $[d]$. We write
$\sigma_u \in [d]$ for the mark of a vertex $u \in \mathcal{T}$.
Let $(\mathcal{T}^{\sigma}, \mathcal{F}_\infty)$ denote the corresponding measurable
space of marked trees, where $\mathcal{F}_k$ is the $\sigma$-algebra generated by the
restriction of the marked tree to its first $k$ generations, and
$\mathcal{F}_\infty = \bigvee_{k \geq 0} \mathcal{F}_k$.
We equip this space with a family of probability measures $(\P_{i})_{i \in [d]}$ such
that, under $\P_{i}$, the process $(\VEC{M}_k(\mathcal{T}, \sigma),\, k \in \N_0)$
associated with the marked tree $(\mathcal{T}, \sigma)$ is a $d$-type BGW process
started from a single individual of type $i$, with offspring distribution
$\P_{i}(\VEC{M}_1 \in \cdot)$.

The size-biased distributions $\P_{i}(\widehat{\VEC{M}}_1 \in \cdot)$, $i \in [d]$, are
defined by
\begin{equation*}
	\P_{i}\!\big(\widehat{\VEC{M}}_1 = \boldsymbol{\ell}\big)
	\coloneqq \P_{i}(\VEC{M}_1 = \boldsymbol{\ell})
	\frac{\boldsymbol{\ell}\cdot \VEC{v}}{\rho\,\VEC{v}(i)},
	\qquad \boldsymbol{\ell} \in \overline{\N}_0^{\,i,\ast},\ i \in [d],
\end{equation*}
where $\rho$ is the maximal eigenvalue of the mean matrix associated with
$( \VEC{M}_k,\, k \in \N_0 )$, computed in~\eqref{eq: derivativef}, and $\VEC{v}$ is
its right eigenvector.
By Proposition~\ref{prop: meanmxrho}, for the $d$-type mutant chain, we have that
$\rho = \frac{rm}{1 - (1-r)m}$ with right eigenvector $\VEC{v} = \VEC{1}$. The
size-biasing therefore simplifies to
\begin{equation*}
	\P_{i}\!\big(\widehat{\VEC{M}}_1 = \boldsymbol{\ell}\big)
	= \frac{|\boldsymbol{\ell}|}{\rho}\,
	\P_{i}(\VEC{M}_1 = \boldsymbol{\ell}),
	\qquad \boldsymbol{\ell} \in \overline{\N}_0^{\,i,\ast},\ i \in [d].
\end{equation*}
Applying the results of~\cite{BigKyp,GeorgiiBaake} together with
Lemma~\ref{lem: eqMand|VECM|} and Proposition~\ref{prop: meanmxrho}, the size-biased
distribution factorises as
\begin{equation}\label{eq:sizebiasedfactorises}
	\begin{split}
		\P_{i}\!\big(\widehat{\VEC{M}}_1 = \boldsymbol{\ell}\big)
		&= \frac{|\boldsymbol{\ell}|}{\rho}\,
		\P_{i}(|\VEC{M}_1| = |\boldsymbol{\ell}|)\,
		\P_{i}\!\big(\VEC{M}_1 = \boldsymbol{\ell} \,\big|\, |\VEC{M}_1| = |\boldsymbol{\ell}|\big) \\
		&= \P_{i}\!\big(\widehat{M}_1 = |\boldsymbol{\ell}|\big)
		\binom{|\boldsymbol{\ell}|}{\boldsymbol{\ell}(1),\dots,\boldsymbol{\ell}(d)}
		\prod_{j \neq i} \VEC{P}_{ij}^{\,\boldsymbol{\ell}(j)},
		\qquad \boldsymbol{\ell} \in \overline{\N}_0^{\,i,\ast},\ i \in [d].
	\end{split}
\end{equation}
That is, $\widehat{\VEC{M}}_1$ under $\P_{i}$ is obtained by size-biasing the offspring
distribution of the underlying mutant-count chain $(|\VEC{M}_k|,\, k \in \N_0)$, then
colouring the offspring according to the multinomial distribution given by the $i$-th
row of the mutation matrix $\VEC{P}$. Recall that $(|\VEC{M}_k|,\, k \in \N_0)$ is
equal in law to the infinite-alleles mutant chain $(M_k,\, k \in \N_0)$ by
Lemma~\ref{lem: eqMand|VECM|}.

For each type $i \in [d]$, we now define a probability measure $\mathbb{Q}_i$ on
$(\mathcal{T}^{\sigma}, \mathcal{F}_\infty)$ under which the marked tree is almost
surely infinite.
The construction proceeds as follows. Start from a single individual
$\gamma_0 = \emptyset$ of type $\eta_0 = i$. This individual produces
$\boldsymbol{\ell}$ offspring with probability
$\P_{\eta_0}(\widehat{\VEC{M}}_1 = \boldsymbol{\ell})$.
Conditionally on $\boldsymbol{\ell}$, each child of type $j$ is selected as the first
vertex of the spine with probability
$\tfrac{\VEC{v}(j)}{\boldsymbol{\ell} \cdot \VEC{v}}$, which reduces to
$\tfrac{1}{|\boldsymbol{\ell}|}$ under $\VEC{v} = \VEC{1}$; that is, the spine
individual $\gamma_1$ is chosen uniformly at random among the children of $\gamma_0$,
and we denote its type by $\eta_1 \in [d]$.
Each of the remaining $|\boldsymbol{\ell}| - 1$ children starts an independent $d$-type
BGW subtree with offspring distribution $\P_{j}(\VEC{M}_1 \in \cdot)$, where $j$ is its
type.
The spine individual $\gamma_1$ then reproduces according to the size-biased law
$\P_{\eta_1}(\widehat{\VEC{M}}_1 \in \cdot)$, and the construction is iterated.
This produces the \emph{spine} $\gamma\coloneqq(\gamma_k)_{k \geq 0}$, a sequence of
distinguished individuals in $\mathcal{T}$ with associated mark sequence
$\eta \coloneqq(\eta_k)_{k \geq 0}$ in $[d]$ defined by
$\eta_k \coloneqq \sigma_{\gamma_k}$. Since size-biased Galton--Watson trees contain an
infinite ray by construction, this procedure yields an almost surely infinite marked
tree.

Let $\widetilde{\mathbb{Q}}_i$ denote the joint law of
$((\mathcal{T}, \sigma); (\gamma, \eta))$. Then $\mathbb{Q}_i$ is defined as the
marginal of $\widetilde{\mathbb{Q}}_i$ on
$(\mathcal{T}^{\sigma}, \mathcal{F}_\infty)$.

\begin{remark}\label{re: spine color}
	Under $\widetilde{\mathbb{Q}}_i$, the spine type sequence $(\eta_k)_{k \geq 0}$ is a
	Markov chain on $[d]$ with transition matrix $\VEC{P}$ and initial state
	$\eta_0 = i$. Indeed, for any $i, j \in [d]$ with $i\neq j$,
	\begin{equation*}
		\begin{split}
			\widetilde{\mathbb{Q}}_i(\eta_1 = j)
			&= \sum_{\boldsymbol{\ell} \in \overline{\N}_0^{\,i,\ast}}
			\P_i(\eta_1 = j \mid \widehat{\VEC{M}}_1 = \boldsymbol{\ell})\,
			\P_i(\widehat{\VEC{M}}_1 = \boldsymbol{\ell}) \\
			&= \sum_{n=1}^{\infty} \P_i(\widehat{M}_1 = n)
			\sum_{\boldsymbol{\ell}:\, |\boldsymbol{\ell}| = n}
			\frac{\boldsymbol{\ell}(j)}{n}\,
			\P_i(\VEC{M}_1 = \boldsymbol{\ell} \mid |\VEC{M}_1| = n) \\
			&= \sum_{n=1}^{\infty} \P_i(\widehat{M}_1 = n)\, \VEC{P}_{ij}
			\;=\; \sum_{n=1}^{\infty} \frac{n\, \P_i(M_1 = n)}{\rho}\, \VEC{P}_{ij}
			\;=\; \VEC{P}_{ij},
		\end{split}
	\end{equation*}
	where the last equality uses $\E_i[M_1] = \rho$. The Markov property follows from the
	iterative construction.
\end{remark}

Under $\widetilde{\mathbb{Q}}_i$, the marked tree $(\mathcal{T}, \sigma)$ is called a
\emph{size-biased multitype BGW tree}. We denote by
$\widehat{\VEC{M}}^{(\gamma_k)}_1$ the size-biased offspring vector produced by the
spine individual $\gamma_k$. Conditionally on the spine $\gamma$ and on the offspring
vectors $(\widehat{\VEC{M}}^{(\gamma_{\ell-1})}_1)_{\ell \geq 1}$ at the spine
vertices, the processes $\{\VEC{M}^{(\ell)}\}_{\ell \geq 1}$ are independent $d$-type
mutant processes with the original (non-size-biased) offspring distributions, started
from
\begin{equation*}
	\VEC{M}^{(\ell)}_0
	\coloneqq \widehat{\VEC{M}}^{(\gamma_{\ell-1})}_1 - \VEC{e}_{\eta_\ell},
\end{equation*}
that is, from the offspring vector of the spine individual $\gamma_{\ell-1}$ with the
spine child $\gamma_\ell$ removed. By construction, the mutant population at the $k$-th
allelic generation decomposes as
\begin{equation}\label{eq: spineprocess}
	\VEC{M}_k = \VEC{e}_{\eta_k} + \sum_{\ell=1}^{k} \VEC{M}^{(\ell)}_{k-\ell}.
\end{equation}
We refer to~\eqref{eq: spineprocess} as the \emph{spine decomposition} of the multitype
BGW process.

As in the one-dimensional case~\cite{LyonsPemantlePeres}, for each $i \in [d]$,
$\mathbb{Q}_i$ is a martingale change of measure from $\P_i$:
\begin{equation*}
	\mathrm{d}\mathbb{Q}_i \big|_{\mathcal{F}_k}
	= W_k \, \mathrm{d}\P_i \big|_{\mathcal{F}_k}, \qquad k \in \N,
\end{equation*}
where $W_k \coloneqq \rho^{-k} |\VEC{M}_k| = \rho^{-k} M_k$ is a non-negative
$\P_i$-martingale. Since $\VEC{v} = \VEC{1}$, the multitype martingale
$\rho^{-k}\, \VEC{M}_k \cdot \VEC{v}$ coincides with the mutant-count chain martingale
$\rho^{-k} \vert \VEC{M}_k \vert$.

\subsection{Application of the spine decomposition: the supercritical case}
\label{sub:spinesupercritical}

Under hypotheses~\eqref{hyp:timextsupc} and~\eqref{hyp:size}, the spine
decomposition yields a probabilistic proof of the Kesten--Stigum theorem. This approach
was introduced by Lyons, Pemantle and Peres~\cite{LyonsPemantlePeres} for the
single-type Galton--Watson process, and provides necessary and sufficient conditions
for the $L^1$-convergence of the martingale defined by the population size normalized
by its mean. The corresponding result for multitype BGW processes is due to Kurtz,
Lyons, Pemantle and Peres~\cite{KLPP}. We adapt it below to the $d$-type mutant chain.

\begin{proposition}[{\cite[Theorems 1 and 2]{KLPP}}] \label{thm: KS-multitype}
	Suppose that~\eqref{hyp:timextsupc} and~\eqref{hyp:size} hold, that is, the
	$d$-type mutant chain is supercritical with finite offspring mean. Then the following
	hold.
	\begin{enumerate}
		\item[(i)] There exists a non-negative random variable $W$ such that, for every
		$i \in [d]$,
		\begin{equation} \label{eq:spine-limit-supercritical}
			\lim_{k \to \infty} \rho^{-k} \VEC{M}_k = W \VEC{u}
			\qquad \P_i\text{-a.s.}
		\end{equation}
		Moreover, $\P_i(W > 0) > 0$ if and only if
		\begin{equation}\label{eq:LlogL}
			\E_i\!\big[ \VEC{M}_1(j) \ln^{+} \VEC{M}_1(j) \big] < \infty
			\quad \text{for all } i, j \in [d].
		\end{equation}
		\item[(ii)] Under~\eqref{eq:LlogL}, on the event of non-extinction
		$\{ \tau_{\VEC{M}} = \infty\}$ and for every $i \in [d]$,
		\[
			\lim_{k \to \infty} \frac{\VEC{M}_k}{|\VEC{M}_k|}
			= \VEC{u} \qquad \P_i\text{-a.s.}
		\]
		That is, the left eigenvector $\VEC{u}$, which is the stationary distribution of
		the mutation matrix $\VEC{P}$, gives the asymptotic proportions of types in the
		population.
	\end{enumerate}
\end{proposition}

\subsection{Application of the spine decomposition: the (sub)critical case}

Under hypothesis~\eqref{hyp:timextsubc} the spine decomposition admits a different
interpretation, related to conditioning on non-extinction. Roelly-Coppoletta and
Rouault~\cite{RRsub} and Geiger~\cite{Geiger99} showed that a (sub)critical single-type
BGW process conditioned on non-extinction can be obtained from the original process by
a martingale change of measure. The resulting process is the $Q$-process and admits a
spine decomposition. We apply this to the mutant-count chain
$(|\VEC{M}_k|,\, k \in \N_0)$.

\begin{proposition}\label{thm: spine subcritical}
	Suppose hypothesis~\eqref{hyp:timextsubc} holds, that is, the $d$-type mutant chain is
	(sub)critical. Then, for each $i \in [d]$, the size-biased distribution of
	$|\VEC{M}_1|$ under $\P_i$ coincides with the offspring distribution of the associated
	$Q$-process:
	\[
		\P_i\bigl(\widehat{|\VEC{M}_1|} \in \cdot\bigr) = \P_i\bigl(|\VEC{M}^{\uparrow}_1| \in \cdot\bigr),
	\]
	where the right-hand side is given by~\eqref{eq: sizebiasedgral} with $q=1$.
	Furthermore, by Corollary~\ref{cor: Conmuta},
	$|\VEC{M}^\uparrow_1| \overset{\mathcal{L}}{=} M^{\uparrow}_1$ under the measure
	$\P_i$. Finally, $\widehat{|\VEC{M}_1|}=|\widehat{\VEC{M}}_1|$ in distribution, since
	the right eigenvector of $\VEC{A}$ is $\VEC{1}$.
\end{proposition}

As a consequence, in the (sub)critical case, the spine decomposition of the $d$-type
mutant chain admits the following description. The spine individual at each allelic
generation reproduces according to the law of $M_1^{\uparrow}$, that is, the
infinite-alleles mutant chain conditioned on non-extinction. One of its children,
chosen uniformly at random, continues the spine, while the remaining children initiate
independent $d$-type mutant processes with the original (non-conditioned) offspring
distributions. At each step, the types of the offspring are assigned according to the
row of $\VEC{P}$ indexed by the parent's type, and the type sequence along the spine is
the Markov chain described in Remark~\ref{re: spine color}.


\appendix

\section{Proof of Proposition \ref{prop: genlawT_0M_1} }\label{app: proofpgf T_0Mvec}
Conditioning on the random vector $\boldsymbol{\xi}^i$ of the offspring of the initial individual of type $i$, equation~\eqref{eq:MGFone} reads
\begin{equation}\label{eq:torre1}
\varphi_{i}( x, \VEC{\bar{y}}^i)
=\sum\limits_{\VEC{v} \in \mathbb{N}_0^d} \boldsymbol{\mu}_i(\VEC{v})\, \mathbb{E}_{\VEC{e}_i}^r\big[ x^{\VEC{T}_0(i)} \prod_{j\neq i}   \VEC{y}(j)^{ \VEC{M}_1(j) }  \Bigm| \boldsymbol{\xi}^{i} = \VEC{v} \big],
\qquad x\in [0,1], \, \VEC{y}\in [0,1]^d.
\end{equation}
It remains to compute the conditional expectation. Given that the initial individual has $\VEC{v}(i)$ children of type $i$, let $(\VEC{T}_{0,1}(i),\VEC{M}_{1,1}) ,\dots,(\VEC{T}_{0,\VEC{v}(i)}(i),\VEC{M}_{1,\VEC{v}(i)})$ be $\VEC{v}(i)$ i.i.d. copies of the random vector $(\VEC{T}_0(i),\VEC{M}_1)$, obtained by the branching property. Then
\begin{align*}
\mathbb{E}_{\VEC{e}_i}^r\left[  x^{\VEC{T}_0(i)} \prod_{j\neq i}   \VEC{y}(j)^{ \VEC{M}_1(j) }
\Bigm| \boldsymbol{\xi}^i = \VEC{v}
\right]
&  = \mathbb{E}_{\VEC{e}_i}^r\big[  {x}^{1 +\sum\limits_{k=1}^{\VEC{v}(i)}\VEC{T}_{0,k}(i)} \prod_{j\neq i}    \VEC{y}(j)^{\VEC{v}(j)+\sum\limits_{k=1}^{\VEC{v}(i)} \VEC{M}_{1,k}(j)}  \big] \\
&  = {x} \big(\varphi_i( {x}, \VEC{\bar{y}}^i )\big)^{\VEC{v}(i)} \prod_{j\neq i} \VEC{y}(j)^{ \VEC{v}(j)}, \qquad \VEC{v} \in \mathbb{N}_0^d.
\end{align*}
 Substituting into~\eqref{eq:torre1} and comparing with~\eqref{eq: g_i} gives~\eqref{eq: pgfTM_i}.

For each $i\in[d]$ and each fixed $\VEC{\bar{y}}^i \in (0,1]^d$, define $g_i^*(z)\coloneqq g_i(z, \VEC{\bar{y}}^i)$.
This function is analytic in a neighborhood of zero, being a probability generating function, and $g_i^*(0) \neq 0$. Indeed, suppose that $g_{i}(0, \VEC{\bar{y}}^i)=0$ for some $\VEC{\bar{y}}^i \in (0,1]^d$. Writing
\[
\beta_k \coloneqq \sum_{\VEC{v} \in \bar{\N}^i_0 } \boldsymbol{\mu}_i\big(\VEC{v}+k\VEC{e}_i\big), \qquad k \in \N_0,
\qquad \text{so that} \qquad
\E_{\VEC{e}_i}^r[ \boldsymbol{\xi}^{i}(i) ] =\sum_{k\in\N} k \beta_k,
\]
the assumption forces $\beta_0=0$, and thus
$
    \E_{\VEC{e}_i}^r[ \boldsymbol{\xi}^{i}(i) ]\geq \beta_1+2(1-\beta_1)=2-\beta_1 .
$
Since assumption~\eqref{hyp:size} implies $ \mathbb{E}_{\VEC{e}_i}^r [ \boldsymbol{\xi}^{i}(i) ]=1-r$, we get $\beta_1 \geq 1+r> 1$, which is impossible.

By the Lagrange inversion formula (see for instance~\cite[Theorem 6.2]{pitman2006combinatorial}), for each $i\in[d]$ and each fixed $\VEC{\bar{y}}^i \in (0,1]^d$, the equation
$
\varphi^*_i(x)= xg^*_i(\varphi^*_i(x))
$
has a unique analytic solution in a neighborhood of $0$, whose expansion in powers of $x$ satisfies
\begin{equation}\label{eq: LIF}
[x^m](\varphi_i(x,\VEC{\bar{y}}^i))^{q} = \frac{q}{m}[x^{m-q}](g_i(x,\VEC{\bar{y}}^i))^m,
\qquad m \geq q \geq 1.
\end{equation}
That is, the coefficient of $x^m$ in $(\varphi_i(x,\VEC{\bar{y}}^i))^{q}$ equals $\frac{q}{m}$ times the coefficient of $x^{m-q}$ in $(g_i(x,\VEC{\bar{y}}^i))^m$.
Indeed, for $x \in (0,1)$ and $\VEC{\bar{y}}^i \in (0,1]^d$ we may write
\[
	{g}_i(x, \VEC{\bar{y}}^i)=\sum\limits_{k=0}^{\infty}a_{i,k}(\VEC{\bar{y}}^i)x^{k}, \qquad \text{where } a_{i,k}(\VEC{\bar{y}}^i)= \sum\limits_{\boldsymbol{\ell} \in \bar{\N}_0^i} \boldsymbol{\mu}_i(\boldsymbol{\ell} + k\VEC{e}_i)
    \prod\limits_{j\neq i} \VEC{y}(j)^{\boldsymbol{\ell}(j)},
\]
and, since $x \mapsto g_i(x,\VEC{\bar{y}}^i)$ is the generating function of the finite measure $\alpha_i = (a_{i,k}(\VEC{\bar{y}}^i) : k \in \mathbb{N}_0)$, the generating function of $\alpha^{*m}_i$ is
\[
	x \mapsto (g_i(x,  \VEC{\bar{y}}^i))^m=
 	\sum\limits_{k=0}^{\infty}x^k \sum\limits_{\boldsymbol{\ell} \in \bar{\N}_0^i }\boldsymbol{\mu}_i^{*m}(\boldsymbol{\ell} + k\VEC{e}_i)\prod\limits_{j\neq i} \VEC{y}(j)^{\boldsymbol{\ell}(j)} .
\]
Hence,  for fixed $\VEC{\bar{y}}^i \in (0,1]^d$ and $x \in (0,1)$,
\begin{equation}\label{eq:pgf TiMexpl}
\varphi_i(x,\VEC{\bar y}^i)
=(\varphi_i(x,\VEC{\bar y}^i))
= \sum\limits_{m=1}^{\infty} \frac{1}{m}\sum\limits_{\boldsymbol{\ell} \in\bar{\N}_0^i}\boldsymbol{\mu}_i^{*m}(\boldsymbol{\ell} + (m-1)\VEC{e}_i)\prod\limits_{j\neq i} \VEC{y}(j)^{\boldsymbol{\ell}(j)}x^m,
\end{equation}
while by definition
\[
\varphi_{i}(x,\VEC{\bar y}^i)
=\sum\limits_{m=1}^{\infty}\sum\limits_{\boldsymbol{\ell} \in \bar{\N}_0^i}\mathbb{P}_{i}^r(\VEC{T}_0(i)=m,\VEC{M}_1=\boldsymbol{\ell})\,x^{m}\prod\limits_{j\neq i}\VEC{y}(j)^{\boldsymbol{\ell}(j)}.
\]
Identifying coefficients yields
\[
	\mathbb{P}_{i}^r(\VEC{T}_0(i)=m,\VEC{M}_1= \boldsymbol{\ell}) = \frac{1}{m}\boldsymbol{\mu}_i^{*m}(\boldsymbol{\ell} + (m-1)\VEC{e}_i), \qquad \boldsymbol{\ell} \in \bar{\N}_0^i .
\]

\hfill $\square$
\section{Proof of Theorem~\ref{thm: Qprocess}}\label{app:proof4}

Throughout, $\VEC{f}_n$ denotes the $n$-th iterate of the offspring generating function
$\VEC{f}$ of $( \VEC{M}_k,\, k\in\N_0 )$, and $\overline{\VEC{f}}_n$ that of
$\overline{\VEC{f}}$, defined in~\eqref{eq barf}. By induction,
$\overline{\VEC{f}}_n(\VEC{s})=\VEC{f}_n(q\VEC{s})/q$ for all $n\in\N_0$.
Recall from Proposition~\ref{prop:barA} that $\overline{\VEC{f}}$ is the generating function
of a (sub)critical positively regular $d$-type BGW process with maximal eigenvalue
$\overline\rho=h'(q)\le 1$, right eigenvector $\VEC{1}$ and left eigenvector $\VEC{u}$,
normalized by $\VEC{u}\cdot\VEC{1}=1$.

 Since
$\VEC{f}(\VEC{q})=\VEC{q}$ implies $\VEC{f}_n(\VEC{q})=\VEC{q}$, for every
$\VEC{y}\in S^\ast$ and $n\in\N$,
\begin{equation}\label{eq:extinctionrepresentation}
	\sum_{\VEC{z}\in S^\ast} P^n_{(\VEC{y},\VEC{z})}\,\VEC{q}^{\VEC{z}}
	= \E_{\VEC{y}}\big[\VEC{q}^{\VEC{M}_n}\big]-\P_{\VEC{y}}(\VEC{M}_n=\VEC{0})
	= \VEC{f}_n^{\VEC{y}}(\VEC{q}) - \VEC{f}_n^{\VEC{y}}(\VEC{0})
	= \VEC{q}^{\VEC{y}} - \VEC{f}_n^{\VEC{y}}(\VEC{0}),
\end{equation}
and moreover
$\VEC{q}^{\VEC{y}} - \VEC{f}_n^{\VEC{y}}(\VEC{0}) = \VEC{q}^{\VEC{y}}\bigl(1-\overline{\VEC{f}}_n^{\VEC{y}}(\VEC{0})\bigr)$.

Fix $k_1<\dots<k_\ell$. For $n\in\N$,
conditioning on $\{ k_\ell+n<\tau_\VEC{M}<\infty \}$ and using the Markov and branching
properties together with~\eqref{eq:extinctionrepresentation} gives
\begin{equation*}
	\begin{split}
		\P_{\VEC{y}_0}( \VEC{M}_{k_1}= \VEC{y}_1, \dots, \VEC{M}_{k_\ell}=\VEC{y}_\ell \mid k_\ell + n < \tau_\VEC{M} <\infty )
		&= \P_{\VEC{y}_0}( \VEC{M}_{k_1}= \VEC{y}_1, \dots, \VEC{M}_{k_\ell}=\VEC{y}_\ell)
		\frac{\P_{\VEC{y}_\ell}( n < \tau_\VEC{M}<\infty)}{\P_{\VEC{y}_0}( k_\ell + n < \tau_\VEC{M}<\infty )}\\
		&= \P_{\VEC{y}_0}( \VEC{M}_{k_1}= \VEC{y}_1, \dots, \VEC{M}_{k_\ell}= \VEC{y}_\ell )
		\frac{\sum_{\VEC{z}\neq \VEC{0}}P^n_{(\VEC{y}_\ell,\VEC{z})}\VEC{q}^{\VEC{z}}}{\sum_{\VEC{z} \neq \VEC{0}}P^{k_\ell + n}_{(\VEC{y}_0,\VEC{z})}\VEC{q}^{\VEC{z}}}.
	\end{split}
\end{equation*}
Using~\eqref{eq:extinctionrepresentation} twice, we write the last ratio as
\begin{equation}\label{eq: fact}
	\begin{split}
		\frac{\sum_{\VEC{z} \neq \VEC{0}}P^n_{(\VEC{y}_\ell,\VEC{z})}\VEC{q}^{\VEC{z}}}{\sum_{\VEC{z} \neq \VEC{0}}P^{k_j + n}_{(\VEC{y}_0,\VEC{z})}\VEC{q}^{\VEC{z}}}
		&= \frac{\VEC{q}^{\VEC{y}_\ell} - \VEC{f}_{n}^{\VEC{y}_\ell}(\VEC{0}) }{\VEC{q}^{\VEC{y}_0} - \VEC{f}_{k_\ell+ n}^{\VEC{y}_0}(\VEC{0}) } \\
		&= \frac{\VEC{q}^{\VEC{y}_0} - \VEC{f}_{n}^{\VEC{y}_0}(\VEC{0}) }{\VEC{q}^{\VEC{y}_0} - \VEC{f}_{k_\ell+ n}^{\VEC{y}_0}(\VEC{0}) }\cdot
		\frac{\VEC{q}^{\VEC{y}_\ell} \bigl(1 - \overline{\VEC{f}}_{n}^{\VEC{y}_\ell}(\VEC{0}) \bigr)}{\VEC{u}\cdot(\VEC{1}- \overline{\VEC{f}}_n(\VEC{0}))}
		\big[\frac{\VEC{q}^{\VEC{y}_0} \bigl(1 - \overline{\VEC{f}}_{n}^{\VEC{y}_0}(\VEC{0})\bigr) }{\VEC{u}\cdot(\VEC{1}- \overline{\VEC{f}}_n(\VEC{0}))}\big]^{-1}.
	\end{split}
\end{equation}

By~\cite[Lemma~5]{nakagawa1978} applied to $\overline{\VEC{f}}$, together with the identity
$\overline{\rho}=h'(q)$ from~\eqref{eq:barrho}, we have
\[
	\lim_{n\rightarrow \infty}\frac{\VEC{q}^{\VEC{y}_0} - \VEC{f}_{n}^{\VEC{y}_0}(\VEC{0}) }{\VEC{q}^{\VEC{y}_0} - \VEC{f}_{n+1}^{\VEC{y}_0}(\VEC{0}) }
	= \lim_{n\rightarrow \infty}\frac{1 - \overline{\VEC{f}}_{n}^{\VEC{y}_0}(\VEC{0}) }{1 - \overline{\VEC{f}}_{n+1}^{\VEC{y}_0}(\VEC{0}) }
	= \frac{1}{\overline{\rho}}=\frac{1}{h'(q)}, \qquad \forall\, \VEC{y}_0 \in S^\ast,
\]
and a recursion shows that
\begin{equation}\label{eq: reclem5}
	\lim_{n\rightarrow \infty} \frac{\VEC{q}^{\VEC{y}_0} - \VEC{f}_{n}^{\VEC{y}_0}(\VEC{0}) }{\VEC{q}^{\VEC{y}_0} - \VEC{f}_{k_\ell+ n}^{\VEC{y}_0}(\VEC{0}) }= \frac{1}{h'(q)^{k_\ell}}.
\end{equation}
By~\cite[Lemma~3]{nakagawa1978}, applied again to $\overline{\VEC{f}}$, and since the right
eigenvector associated with $\overline\rho$ is $\VEC{1}$,
\begin{equation}\label{eq: lemma3}
	\lim_{n\rightarrow \infty}
	\frac{\overline{\VEC{f}}_{n}^{\VEC{z}}(\VEC{1}) - \overline{\VEC{f}}_{n}^{\VEC{z}}(\VEC{0}) }{\VEC{u} \cdot(\overline{\VEC{f}}_n(\VEC{1})- \overline{\VEC{f}}_n(\VEC{0}))}
	=\lim_{n\rightarrow \infty}\frac{1 - \overline{\VEC{f}}_{n}^{\VEC{z}}(\VEC{0}) }{\VEC{u}\cdot(\VEC{1}-\overline{\VEC{f}}_n(\VEC{0}))}
	=\VEC{z}\cdot \VEC{1}=|\VEC{z}|, \qquad \forall\, \VEC{z} \in S^\ast.
\end{equation}
Applying~\eqref{eq: reclem5} and~\eqref{eq: lemma3} in~\eqref{eq: fact} gives
\begin{equation}\label{eq: limratnstepgmM}
	\lim_{n\rightarrow\infty} \frac{\sum_{\VEC{z} \neq \VEC{0}}P^n_{(\VEC{y}_\ell,\VEC{z})}\VEC{q}^{\VEC{z}}}{\sum_{\VEC{z} \neq \VEC{0}}P^{k_\ell + n}_{(\VEC{y}_0,\VEC{z})}\VEC{q}^{\VEC{z}}}
	= \frac{1}{h'(q)^{k_\ell}} \frac{\VEC{q}^{\VEC{y}_\ell}\,\VEC{y}_\ell \cdot \VEC{1}}{\VEC{q}^{\VEC{y}_0}\,\VEC{y}_0 \cdot \VEC{1}}
	=\frac{1}{h'(q)^{k_\ell}} \frac{q^{|\VEC{y}_\ell|}|\VEC{y}_\ell |}{q^{|\VEC{y}_0|}|\VEC{y}_0|}.
\end{equation}
Substituting back into the initial conditioning identity yields~\eqref{eq:Qlim}. The
Markov property and the expression~\eqref{eq: Qtransitions} for the transition
probabilities follow by taking ratios of consecutive finite-dimensional distributions,
exactly as in the proof of Theorem~\ref{thm:condnonext_clonemut}.

Finally, under hypothesis~\eqref{hyp:timextsubc} we have $q=1$, so
$\overline{\VEC{f}}=\VEC{f}$, $\overline\rho=\rho=h'(1)$ and the conditioning event reduces to
$\{\tau_\VEC{M}>k_j+n\}$; \eqref{eq:Qlim} then specialises to~\eqref{eq:Qlimsub} by
Proposition~\ref{prop: meanmxrho}.
\qed


\section*{Acknowledgements}
 MCF’s research is supported by the UNAM-PAPIIT grant IN109924.
SHT acknowledges the support of the UNAM-PAPIIT grant IA106126.

\bibliographystyle{plain}

\bibliography{biblio}
\end{document}